\documentclass[12pt]{amsart}

\usepackage{amsthm,amssymb}
\usepackage[colorlinks,linkcolor=red,citecolor=blue,urlcolor=blue,hypertexnames=true]{hyperref}

\input{xy}
\xyoption{all}

\newcommand{\Max}{{\rm Max}}
\DeclareMathOperator{\cofi}{hocofiber}
\newcommand{\tor}{\mathrm{Tor}}
\newcommand{\st}{\mathrm{st}}
\newcommand{\Idem}{\mathrm{Idem}}
\newcommand{\GL}{\mathrm{GL}}
\newcommand{\Spec}{\mathrm{Spec}}
\DeclareMathOperator{\hofi}{hofiber}
\newcommand{\spec}{{\rm Spec}\,}
\newcommand{\comp}{\mathfrak{Comp}}
\newcommand{\pol}{\mathfrak{Pol}}
\newcommand{\set}{\mathfrak{Set}}
\newcommand{\rC}{\mathfrak{C}}
\newcommand{\fC}{\mathfrak{A}}
\newcommand{\frakm}{\mathfrak{M}}
\newcommand{\rA}{\mathfrak{Ass}}
\newcommand{\Spt}{\mathfrak{Spt}}

\newcommand{\Co}{\mathfrak{Comm}}
\newcommand{\ab}{\mathfrak{Ab}}
\newcommand{\cF}{\mathcal F}
\newcommand{\cO}{\mathcal O}
\newcommand{\cV}{\mathcal V}
\newcommand{\N}{\mathbb N}
\newcommand{\C}{\mathbb C}
\newcommand{\R}{\mathbb R}
\newcommand{\F}{\mathbb F}
\newcommand{\Q}{\mathbb Q}
\newcommand{\Z}{\mathbb Z}
\newcommand{\cod}{{\rm cod}}
\newcommand{\map}{{\rm map}}
\def\colim{{\operatornamewithlimits{colim}}}
\newcommand{\coker}{{\rm coker}\,}
\newcommand{\comment}[1]{}

\theoremstyle{plain}
\newtheorem{theorem}[equation]{Theorem}
\newtheorem{introth}[equation]{Theorem}
\newtheorem{corollary}[equation]{Corollary}
\newtheorem{proposition}[equation]{Proposition}
\newtheorem{lemma}[equation]{Lemma}

\theoremstyle{definition}
\newtheorem{definition}[equation]{Definition}

\theoremstyle{remark}
\newtheorem{remark}[equation]{Remark}
\newtheorem{example}[equation]{Example}
\newtheorem{examples}[equation]{Examples}
\newtheorem*{ack}{Acknowledgements}
\begin{document}

\title{Algebraic Geometry of Topological Spaces I}

%\subjclass{13D15, 13C10, 46J10}

\author{Guillermo Corti\~nas}
\address{Guillermo Corti\~nas, Departamento de Matem\'atica\\ FCEN-Universidad de Buenos Aires\\ Ciudad Universitaria Pab 1\\ (1428) Buenos Aires, Argentina}
\email{gcorti@dm.uba.ar}
\urladdr{http://mate.dm.uba.ar/\~{}gcorti}

\author{Andreas Thom}
\address{Andreas Thom, Mathematisches Institut der Universit\"at Leipzig,
Johannisgasse 26\\ 04103 Leipzig, Germany.}
\email{thom@math.uni-leipzig.de}
\urladdr{http://www.math.uni-leipzig.de/MI/thom}
\thanks{Corti\~nas' research was partly supported by grants PICT 2006-00836, UBACyT-X057, and
MTM2007-64704. Thom's research was partly supported by the DFG (GK \textit{Gruppen und Geometrie} G\"ottingen).}

\begin{abstract}
We use techniques from both real and complex algebraic geometry to study $K$-theoretic and related invariants of
the algebra $C(X)$ of continuous complex-valued functions on a compact Hausdorff topological space $X$. For example,
we prove a parametrized version of a theorem of Joseph Gubeladze; we show that if $M$ is a countable, abelian, cancellative, torsion-free, seminormal monoid, and $X$ is contractible, then every finitely generated projective module over $C(X)[M]$ is free. The particular case $M=\N_{0}^n$ gives a parametrized version of the celebrated theorem proved independently by Daniel Quillen and Andrei Suslin that finitely generated projective modules over a polynomial ring over a field are free. The conjecture of Jonathan Rosenberg which
predicts the homotopy invariance of the negative algebraic $K$-theory of $C(X)$ follows from the particular case $M=\Z^n$. We also give algebraic conditions for a functor from commutative algebras to abelian groups to be homotopy invariant on $C^*$-algebras, and for a homology theory of commutative algebras to vanish on $C^*$-algebras. These criteria have numerous applications. For example, the vanishing criterion applied to nil-$K$-theory implies that commutative $C^*$-algebras are $K$-regular. As another application, we show that the familiar formulas of
Hochschild-Kostant-Rosenberg and Loday-Quillen for the algebraic Hochschild and cyclic homology of the coordinate ring of a smooth algebraic variety remain valid
for the algebraic Hochschild and cyclic homology of $C(X)$. Applications to the conjectures of Be\u\i linson-Soul\'e and Farrell-Jones are also given.
\end{abstract}

\maketitle

\tableofcontents

\section{Introduction.}\label{sec:intro}
\numberwithin{equation}{section}
In his foundational paper \cite{fac}, Jean-Pierre Serre asked whether all finitely generated projective modules over the polynomial ring $k[t_1,\dots,t_n]$ over a field $k$ are free. This question, which became known as Serre's conjecture, remained open for about twenty years. An affirmative answer was given independently by Daniel Quillen \cite{quiqs} and Andrei Suslin \cite{susqs}. Richard G. Swan observed in \cite{swanlaurent} that the Quillen-Suslin theorem implies
that all finitely generated projective modules over the Laurent polynomial ring $k[t_1,t_1^{-1}\dots,t_n,t_n^{-1}]$ are free.
This was later generalized by Joseph Gubeladze \cite{gubelan1}, \cite{gubelan2}, who proved, among other things, that if $M$ is an abelian, cancellative, torsion-free, seminormal monoid, then every finitely generated projective module over $k[M]$ is free. Quillen-Suslin's theorem and Swan's theorem are the special cases $M=\N_0^n$ and $M=\Z^n$ of Gubeladze's result. On the other hand, it is classical that if $X$ is a contractible compact Hausdorff space, then all finitely generated projective modules over the algebra $C(X)$ of complex-valued continuous functions on $X$ -- which by another theorem of Swan, are the same thing as locally trivial complex vector bundles on $X$ -- are free. In this paper we prove (see \ref{thm:pswan}):

\begin{introth}\label{thm:intropgubel}
Let $X$ be a contractible compact space, and $M$ a countable, cancellative, torsion-free, seminormal, abelian monoid. Then every finitely generated projective module over $C(X)[M]$ is free.
\end{introth}

Moreover we show (Theorem \ref{thm:loctriv}) that bundles of finitely generated free $\C[M]$-modules over a not necessarily contractible, compact Hausdorff $X$ which are direct summands of trivial bundles, are locally trivial.
The case $M=\N_0^n$ of Theorem \ref{thm:intropgubel} gives a parametrized version of Quillen-Suslin's theorem. The case $M=\Z^n$ is connected with a conjecture of Jonathan Rosenberg \cite{roshan} which predicts that the negative algebraic $K$-theory groups of $C(X)$ are homotopy invariant for compact Hausdorff $X$. Indeed, if $R$ is any ring, then
the negative algebraic $K$-theory group $K_{-n}(R)$ is defined as a certain canonical direct summand of $K_0(R[\Z^n])$;
the theorem above thus implies that $K_{-n}(C(X))=0$ if $X$ is contractible. Using this and excision, we derive
the following result (see Theorem \ref{thm:rosen1}).

\begin{introth}\label{thm:introrosconj}
Let $\comp$ be the category of compact Hausdorff spaces and let $n>0$. Then the functor $\comp\to \ab$,
$X\mapsto K_{-n}(C(X))$ is homotopy invariant.
\end{introth}

A partial result in the direction of Theorem \ref{thm:introrosconj} was obtained by Eric Friedlander and Mark E. Walker in \cite{fw}. They proved that $K_{-n}(C(\Delta^p))=0$ for $p\ge 0$, $n>0$. In Section \ref{subsec:second},
we give a second proof of Theorem \ref{thm:introrosconj} which uses the Friedlander-Walker result. Elaborating
on their techniques, and combining them with our own methods, we obtain the following general criterion for homotopy invariance (see Theorem \ref{thm:htpy}).

\begin{introth} \label{thm:introhtpy}
Let $F$ be a functor on the category $\Co/\C$ of commutative $\C$-algebras with values in the category $\ab$ of abelian groups. Assume that the following three conditions are satisfied.
\begin{itemize}
\item[(i)] $F$ is split-exact on $C^*$-algebras.
\item[(ii)] $F$ vanishes on coordinate rings of smooth affine varieties.
\item[(iii)] $F$ commutes with filtering colimits.\\
\end{itemize}
\vspace{0.0125cm}

Then the functor
\[
\comp\to\ab,\quad X \mapsto F(C(X)),
\]
is homotopy invariant and $F(C(X))=0$ for $X$ contractible.
\end{introth}

Observe that $K_{-n}$ satisfies all the hypothesis of the theorem above $(n>0)$. This gives a third proof of Theorem \ref{thm:introrosconj}. We also use Theorem \ref{thm:introhtpy} to prove the following vanishing theorem for homology theories (see \ref{thm:vanish}). In this paper a homology theory on a category $\rC$ of algebras is simply a functor $E:\rC \to \Spt$ to the category of spectra which preserves finite products up to homotopy.

\begin{introth}\label{thm:introvanish}
Let $E:\Co/\C\to \Spt$ be a homology theory of commutative $\C$-algebras and $n_0\in\Z$. Assume that the following three conditions are satisfied.
\begin{itemize}
\item[(i)] $E$ satisfies excision on commutative $C^*$-algebras.
\item[(ii)] $E_n$ commutes with filtering colimits for $n\ge n_0$.
\item[(iii)] $E_n(\cO(V))=0$ for each smooth affine algebraic variety $V$ for $n\ge n_0$.\\
\end{itemize}
\vspace{0.0125cm}

Then $E_n(A)=0$ for every commutative $C^*$-algebra $A$ for $n\ge n_0$.
\end{introth}

Recall that a ring $R$ is called $K$-regular if
\[
\coker(K_n(R)\to K_n(R[t_1,\dots,t_p]))=0\qquad (p\ge 1, n\in\Z).
\]
As an application of Theorem \ref{thm:introvanish} to the homology theory $F^p(A)=\cofi(K(A\otimes_\C\cO(V))\to K(A\otimes_\C\cO(V)[t_1,\dots,t_p])$ where $V$ is a smooth algebraic variety, we obtain the following (Theorem
\ref{thm:kreg}).

\begin{introth}\label{thm:introkreg}
Let $V$ be a smooth affine algebraic variety over $\C$, $R=\cO(V)$, and $A$ a commutative $C^*$-algebra. Then $A\otimes_\C R$ is $K$-regular.
\end{introth}

The case $R=\C$ of the previous result was discovered by Jonathan Rosenberg, see Remark \ref{rem:creditreg}.
%Unfortunately, the two proofs he has given, in \cite[Thm. 3.1]{rosop} and \cite[p.\ 866]{roshan}, have gaps. A version
%of Theorem \ref{thm:introkreg} for $A=C(D)$, $D$ a compact polyhedron, was given by Friedlander and Walker in \cite[Thm.\ 5.3]{fw}.
%An alternative proof of Theorem \ref{thm:introkreg} is to deduce the general case from the latter result of Friedlander-Walker; this can be done using Theorem \ref{thm:draw}, or using the argument sketched by Rosenberg in \cite[p. 91]{rosop}.

We also give an application of Theorem \ref{thm:introvanish} which concerns the algebraic Hochschild and cyclic homology of $C(X)$.
We use the theorem in combination with the celebrated results of Gerhard Hochschild, Bertram Kostant and Alex Rosenberg (\cite{hkr}) and of Daniel Quillen and Jean-Louis Loday (\cite{lq}) on the Hochschild and cyclic homology
of smooth affine algebraic varieties and the spectral sequence of Christian Kassel and Arne Sletsj{\o}e (\cite{ks}), to prove the following (see Theorem \ref{thm:hh} for a full statement of our result and for the appropriate definitions).

\begin{introth}\label{thm:introhh}
Let $k\subset\C$ be a subfield. Write $HH_*(/k)$, $HC_*(/k)$, $\Omega^*_{/k}$, $d$ and $H_{dR}^*(/k)$ for algebraic  Hochschild and
cyclic homology, algebraic K\"ahler differential forms, exterior differentiation, and algebraic de Rham cohomology, all taken relative to the field $k$. Let $X$ be a compact Hausdorff space. Then
\begin{gather*}
HH_n(C(X)/k)=\Omega^n_{C(X)/k}\\
HC_n(C(X)/k)=\Omega^n_{C(X)/k}/d\Omega^{n-1}_{C(X)/k}\oplus\bigoplus_{2\le 2p\le n}H_{dR}^{n-2p}(C(X)/k)\qquad (n\in\Z)
\end{gather*}
\end{introth}

We also apply Theorem \ref{thm:introvanish} to the $K$-theoretic isomorphism conjecture of Farrell-Jones and to the Beilison-Soul\'e conjecture. The $K$-theoretic isomorphism conjecture for the group $\Gamma$ with coefficients in a ring $R$ asserts that a certain assembly map
\[
\mathcal{A}^\Gamma(R):\mathbb{H}^{\Gamma}(E_{\mathcal{VC}}(\Gamma),K(R))\to K(R[\Gamma])
\]
is an equivalence. Applying Theorem \ref{thm:introvanish} to the cofiber of the assembly map, we obtain that if
$\mathcal{A}^\Gamma(\cO(V))$ is an equivalence for each smooth affine algebraic variety $V$ over $\C$, then
$\mathcal{A}^\Gamma(A)$ is an equivalence for any commutative $C^*$-algebra $A$. The (rational) Be\u\i linson-Soul\'e
conjecture concerns the decomposition of the rational $K$-theory of a commutative ring into the sum of eigenspaces
of the Adams operations
\[
K_n(R)\otimes\Q=\oplus_{i\ge 0}K_n^{(i)}(R)
\]
The conjecture asserts that if $R$ is regular noetherian, then
\[
K_n^{(i)}(R)=0 \mbox{ for } n\ge \max\{1,2i\}
\]
It is well-known that the validity of the conjecture for $R=\C$ would imply that it also holds for $R=\cO(V)$
whenever $V$ is a smooth algebraic variety over $\C$. We use Theorem \ref{thm:introvanish} to show that the validity of the conjecture for $\C$ would further imply that it holds for every commutative $C^*$-algebra.

\vspace{0.2cm}

Next we give an idea of the proofs of our main results, Theorem \ref{thm:intropgubel} and Theorem \ref{thm:introhtpy}.

The basic idea of the proof of Theorem \ref{thm:intropgubel} goes back to Rosenberg's article \cite{roskk} and ultimately to the usual proof of the fact that locally trivial bundles over a contractible compact Hausdorff space
are trivial. It consists of translating the question of the freedom of projective modules into a lifting problem:
\begin{equation}\label{diag:introlift}
\xymatrix{& &\GL(\C[M])\ar[d]^\pi\\
X\ar[r]^e\ar@{.>}[urr] &P_n(\C[M])&\frac{\GL(\C[M])}{\GL_{[1,n]}(\C[M])\times \GL_{[n+1,\infty)}(\C[M])}\ar[l]_(.6)\iota}
\end{equation}

Here we think of a projective module of constant rank $n$ over $C(X)$ as a map $e$ to the set of all rank $n$ idempotent matrices, which by Gubeladze's theorem is the same as the set $P_n(\C[M])$ of those matrices which are
conjugate to the diagonal matrix $1_n\oplus 0_\infty$. Thus $g\mapsto g(1_n\oplus 0_\infty)g^{-1}$ defines a surjective map $\GL(\C[M])\to P_n(\C[M])$ which identifies the latter set with the quotient of $\GL(\C[M])$ by the
stabilizer of $1_n\oplus 0_\infty$, which is precisely the subgroup $\GL_{[1,n]}(\C[M])\times \GL_{[n+1,\infty)}(\C[M])$. For this setup to make sense we need to equip each set involved in \eqref{diag:introlift} with a topology in such a way that all maps in the diagram are continuous. Moreover for the lifting problem to have a solution, it will suffice to show that $\iota$ is a homeomorphism and that $\pi$ is a \emph{compact fibration}, i.e. that it restricts to a fibration over each compact subset of the base. In Section
\ref{sec:fib} we show that any countable dimensional $\R$-algebra $R$ is equipped with a canonical compactly generated topology which makes it into a topological algebra. A subset $F\subset R$ is closed in this topology if and only if $F\cap B\subset B$ is closed for every compact semi-algebraic subset $B$ of every finite dimensional subspace of $R$. In particular this applies to $M_\infty R$. The subset $P_n(R)\subset M_\infty(R)$ carries
the induced topology, and the map $e$ of \eqref{diag:introlift} is continuous for this topology. The group
$\GL (R)$ also carries a topology, generated by the compact semi-algebraic subsets $\GL_n(R)^B$. Here $B\subset M_nR$ is any compact semi-algebraic subset as before, and $\GL_n(R)^B$ consists
of those $n\times n$ invertible matrices $g$ such that both $g$ and $g^{-1}$ belong to $B$. The subgroup $\GL_{[1,n]}(R)\times \GL_{[n+1,\infty)}(R)\subset \GL(R)$ turns out to be closed, and we show in Section \ref{subsec:quot} -- with the aid of Gregory Brumfiel's theorem on quotients of semi-algebraic sets (see \ref{thm:brum}) -- that for a topological group $G$ of this
kind, the quotient $G/H$ by a closed subgroup $H$ is again compactly generated by the images of the compact semi-algebraic subsets defining the topology of $G$, and these images are again compact, semi-algebraic subsets. Moreover the restriction of the projection $\pi:G\to G/H$ over each compact semi-algebraic subset $S\subset G$
is semi-algebraic. We also show (Theorem \ref{thm:fib}) that $\pi$ is a compact fibration.
This boils
down to showing that if $S\subset G$ is compact semi-algebraic, and $T=f(S)$, then we can find an open covering of $T$
such that $\pi$ has a section over each open set in the covering. Next we observe that if $U$ is any space, then the group $\map(U,G)$ acts on the set $\map(U,G/H)$, and a map $U\to G/H$ lifts to $U\to G$ if and only if its
class in the quotient $F(U)=\map(U,G/H)/\map(U,G)$ is the class of the trivial element: the constant map $u\mapsto H$ ($u\in U$). For example the class of the composite of $\pi$ with the inclusion $S\subset G$ is the trivial element of $F(S)$. Hence if $p=\pi_{|S}:S\to T$, then $F(p)$ sends the inclusion $T\subset G/H$ to the trivial element of $F(S)$. In Section \ref{sec:splitex} we introduce a notion of (weak) split exactness for contravariant functors of topological spaces with values in pointed sets; for example the functor $F$ introduced above is split exact (Lemma \ref{splitlem}). The key
technical tool for proving that $\pi$ is a fibration is the following (see \ref{thm:main}); its proof
uses the good topological properties of semi-algebraic sets and maps, especially Hardt's triviality theorem \ref{thm:hardt}.

\begin{introth}\label{thm:intromain}
Let $T$ be a compact semi-algebraic subset of $\R^k$.
Let $S$ be a semi-algebraic set and let $f\colon S \to T$ be a proper continuous semi-algebraic surjection.  Then, there exists a semi-algebraic triangulation of $T$ such that for every weakly split-exact contravariant functor $F$ from the category $\pol$ of compact polyhedra to the category $\set_+$ of pointed sets, and every simplex $\Delta^n$ in the
triangulation, we have
\[
\ker(F(\Delta^n)\to F(f^{-1}(\Delta^n)))=*
\]
\end{introth}

Here $\ker$ is the kernel in the category of pointed sets, i.e. the fiber over the base point. In our situation Theorem \ref{thm:intromain} applies to show that there is a triangulation of $T\subset G/H$
such that the projection $\pi$ has a section over each simplex in the triangulation. A standard argument now shows that $T$ has an open covering (by open stars of a subdivision of the previous triangulation) such that $\pi$ has
 section over each open set in the covering.
Thus in diagram
\eqref{diag:introlift} we have that $\pi$ is a compact fibration and that $e$ is continuous. The map
$\iota:\GL(R)/\GL_{[1,n]}(R)\times \GL_{[n+1,\infty)}(R)\to P_n(R)$ is continuous for every countable dimensional $\R$-algebra $R$ (see \ref{subsec:algcomp}). We show in Proposition \ref{prop:algcompeq} that it is a homeomorphism whenever the map
\begin{equation}\label{map:bounded}
K_0(\ell^\infty(R))\to \prod_{n\ge 1}K_0(R)
\end{equation}
is injective. Here $\ell^\infty(R)$ is the set of all sequences $\N\mapsto R$ whose image is contained in one of the compact semi-algebraic subsets $B\subset R$ which define the topology of $R$; it is isomorphic to $\ell^\infty(\R)\otimes R$ (Lemma \ref{lem:bounded}). The algebraic compactness theorem
(\ref{thm:algcomp}) says that if $R$ is a countable dimensional $\C$-algebra such that

\begin{equation}\label{map:introalgcomp}
K_0(\cO(V))\overset\sim\longrightarrow K_0(\cO(V)\otimes_\C R)\qquad (\forall \mbox{ smooth affine } V),
\end{equation}

then \eqref{map:bounded} is injective. A theorem of Swan (see \ref{thm:swan}) implies that $R=\C[M]$ satisfies
\eqref{map:introalgcomp}. Thus the map $\iota$ of diagram \eqref{diag:introlift} is a homeomorphism. This concludes
the sketch of the proof of Theorem \ref{thm:intropgubel}.

The proof of the algebraic compactness theorem uses the following theorem (see \ref{thm:boundedapprox}).

\begin{introth}\label{thm:introboundedapprox}
Let $F$ and $G$ be functors from commutative $\C$-algebras to sets.
Assume that both $F$ and $G$ preserve filtering colimits. Let $\tau:F\to G$ be a natural transformation.
Assume that $\tau(\cO(V))$ is injective (resp. surjective) for each smooth affine algebraic variety $V$ over $\C$.
Then $\tau(\ell^\infty(\C))$ is injective (resp. surjective).
\end{introth}

The proof of \ref{thm:introboundedapprox} uses a technique which we call \emph{algebraic approximation}, which we
now explain. Any commutative $\C$-algebra is the colimit of its subalgebras of finite type, which form
a filtered system. If the algebra contains no nilpotent elements, then each of
its subalgebras of finite type is of the form $\cO(Y)$ for {\it affine variety} $Y$,
by which we mean a reduced affine scheme of finite type over $\C$. If $\C[f_1,\dots,f_n]\subset \ell^\infty(\C)$ is the subalgebra generated by $f_1,\dots,f_n$, and $\C[f_1,\dots,f_n]\cong \cO(Y)$, then $Y$ is isomorphic
to a closed subvariety of $\C^n$, and $f=(f_1,\dots,f_n)$ defines a map from $\N$ to a precompact subset of the
 space $Y_{an}$ of closed points of $Y$ equipped the topology inherited by the euclidean topology on $\C^n$.
The space $Y_{an}$ is equipped with the structure of a (possibly singular) analytic variety, whence
the subscript. Summing up, we have
\begin{equation}\label{eq:filtering}
\ell^\infty(\C)=\colim_{\N\to Y_{an}}\cO(Y)
\end{equation}
where the colimit runs over all affine varieties $Y$ and all maps with precompact image. The proof of Theorem
\ref{thm:introboundedapprox} consists of showing that in \eqref{eq:filtering} we can restrict to maps $\N\to V$ with $V$
smooth. This uses Hironaka's desingularization \cite{hironaka} to lift a map $f:\N\to V$ with $V$ affine and singular, to a map
$f':\N\to \tilde{V}$ with $\tilde{V}$ smooth and possibly non-affine, and Jouanoulou's device \cite{jou} to further lift $f'$ to a map $f'':\N\to W$ with $W$ smooth and affine.

The idea of algebraic approximation appears in the work of Jonathan Rosenberg \cite{roskk, rosop, roshan}, and later in the article of Eric Friedlander and Mark E. Walker \cite{fw}. One source of inspiration is the work of Andrei Suslin \cite{MR714090}. In \cite{MR714090}, Suslin studies an inclusion of algebraically closed fields $L \subset K$ and analyzes $K$ successfully in terms of its finitely generated $L$-subalgebras.

Next we sketch the proof of Theorem \ref{thm:introhtpy}. The first step is to reduce to the polyhedral case. For
this we use Theorem \ref{thm:introdraw} below, proved in \ref{thm:draw}. Its proof uses another algebraic approximation argument
%(see the proof of Lemma \ref{lem:algcolimit}),
together with a result of Allan Calder and Jerrold Siegel, which says that the right Kan extension to $\comp$ of a homotopy invariant functor defined on $\pol$ is homotopy invariant on $\comp$.

\begin{introth}\label{thm:introdraw}
Let $F:\Co\to\ab$ be a functor. Assume that $F$ satisfies each of the following conditions.
\begin{enumerate}
\item[(i)] $F$ commutes with filtered colimits.
%\item[(ii)] $F$ is split-exact on $C^*$-algebras.
\item[(iii)]  The functor $\pol\to\ab$, $D \mapsto F(C(D))$ is homotopy invariant.
\end{enumerate}
Then the functor
\[
\comp\to\ab,\quad X \mapsto F(C(X))
\]
is homotopy invariant.
\end{introth}

Next, Proposition \ref{prop:homcor} says that we can restrict to showing that $F$ vanishes on contractible polyhedra.
Since any contractible polyhedron is a retract of its cone, which is a starlike polyhedron, we further reduce to showing
that $F$ vanishes on starlike polyhedra. Using excision, we may restrict once more, to proving that $F(\Delta^p)=0$ for all $p$.  For this we follow the strategy used by Friedlander-Walker in \cite{fw}. To start, we use algebraic approximation again.  We write
\begin{equation}\label{eq:introapprox}
C(\Delta^p)=\colim_{\Delta^p\to Y_{an}}\cO(Y)
\end{equation}
where the colimit runs over all continuous maps from $\Delta^p$ to
affine algebraic varieties, equipped with the euclidean topology.
Since $F$ is assumed to vanish on $\cO(V)$ for smooth affine $V$, it
would suffice to show that any map $\Delta^p\to Y_{an}$ factors as
$\Delta^p\to V_{an}\to Y_{an}$ with $V$ smooth and affine. Actually
using excision again we may restrict to showing this for each
simplex in a sufficiently fine triangulation of $\Delta^p$. As in
the proof of Theorem \ref{thm:intropgubel}, this is done using
Hironaka's desingularization, Jouanoulou's device and Theorem
\ref{thm:intromain}.

\goodbreak

\bigskip
The rest of this paper is organized as follows. In Section \ref{sec:splitex} we give the appropriate definitions
and first properties of split exactness. We also recall some facts about algebraic $K$-theory and cyclic homology,
such as the key results of Andrei Suslin and Mariusz Wodzicki on excision for algebraic $K$-theory and algebraic
cyclic homology. In Section \ref{sec:real}, we recall some facts from real algebraic geometry, and
prove Theorem \ref{thm:intromain} (\ref{thm:main}). Large semi-algebraic groups
and their associated compactly generated topological groups are the subject of Section \ref{sec:fib}. The main result of this section is the Fibration Theorem \ref{thm:fib} which says that the quotient map of such a group by a closed subgroup is a compact fibration. Section \ref{sec:algcomp} is devoted to algebraic compactness,
that is, to the problem of giving conditions on a countable dimensional algebra $R$ so that the map
$\iota:\GL(R)/\GL_{[1,n]}(R)\times \GL_{[n+1,\infty)}(R)\to P_n(R)$ be a homeomorphism. The connection between
this problem and the algebra $\ell^\infty(R)$ of bounded sequences is established by Proposition \ref{prop:algcompeq}. Theorem \ref{thm:introboundedapprox} is proved in \ref{thm:boundedapprox}. Theorem \ref{thm:algcomp} establishes that the map \eqref{map:bounded} is injective whenever \eqref{map:introalgcomp} holds.
Section \ref{sec:pgubel} contains the proofs of Theorems \ref{thm:intropgubel} and \ref{thm:introrosconj} (\ref{thm:pswan} and \ref{thm:rosen1}). We also show (Theorem \ref{thm:loctriv}) that if $M$ is a monoid as in Theorem \ref{thm:intropgubel} then any bundle of finitely generated free $\C[M]$ modules over a compact
Hausdorff space which is a direct summand of a trivial bundle is locally trivial. Section \ref{sec:htpy} deals
with homotopy invariance. Theorems \ref{thm:introdraw}, \ref{thm:introhtpy} and \ref{thm:introvanish} (\ref{thm:draw}, \ref{thm:htpy} and \ref{thm:vanish}) are proved in this section, where also a second proof
of Rosenberg's conjecture, using a result of Friedlander and Walker, is given (see \ref{subsec:second}).
Section \ref{sec:appvanish} is devoted to applications of the homotopy invariance and vanishing homology theorems,
including Theorems \ref{thm:introkreg} and \ref{thm:introhh} (\ref{thm:kreg} and \ref{thm:hh}) and also to the applications to the conjectures of Farrell-Jones (\ref{thm:fj}, \ref{thm:fjg}) and of Be\u\i linson-Soul\'e (\ref{thm:bs}).

\section{Split-exactness, homology theories, and excision.}\label{sec:splitex}
\numberwithin{equation}{subsection}
\subsection{Set-valued split-exact functors on the category of compact Hausdorff spaces.} \label{subsec:splitsec}

In this section we consider contravariant functors from the category of compact Hausdorff topological spaces to the category $\set_+$ of pointed sets. Recall that if $T$ is a pointed set and $f:S\to T$ is a map, then
\[
\ker f=\{s\in S:f(s)=*\}
\]
We say that a functor $F\colon \comp \to \set_+$ is \emph{split exact} if
for each push-out square
\begin{equation}\label{diag:pushout}
\xymatrix{ X_{12} \ar^{\iota_1}[r] \ar^{\iota_2}[d] & X_1 \ar[d] \\ X_2 \ar[r]& X }
\end{equation}
of topological spaces with $\iota_1$ or $\iota_2$ split injective, the map
\[F(X) \to F(X_1) \times_{F(X_{12})} F(X_2)\]
is a surjection with trivial kernel. We say that $F$ is \emph{weakly split exact} if the map above has
trivial kernel. In case the functor takes values in abelian groups, the notion of split exactness above is equivalent to the usual one. For more details on split-exact functors taking values in the category $\ab$ of abelian groups, see Subsection \ref{subsec:splitexab}.

\vspace{0.2cm}

In the next lemma and elsewhere, if $X$ and $Y$ are topological spaces, we write
\[
\map(X,Y)=\{f:X\to Y \mbox{ continuous} \}
\]
for the set of continuous maps from $X$ to $Y$.

\begin{lemma}
Let $(Y,y)$ be a pointed topological space. The contravariant functor
$$X \mapsto \map(X,Y)$$
from compact Hausdorff topological spaces to pointed sets is split-exact.
\end{lemma}
\begin{proof}
Note that $\map(X,Y)$ is naturally pointed by the constant map taking the value $y \in Y$.
Let \eqref{diag:pushout}
be a push-out of compact Hausdorff topological spaces and assume that $\iota_1$ is a split-injection. It is sufficient to show that
the diagram
$$ \xymatrix{ \map(X_{12},Y)   & \map(X_1,Y) \ar[l] \\  \map(X_2,Y) \ar[u] & \map(X,Y) \ar[u] \ar[l]} $$
is a pull-back. But this is immediate from the universal property of a push-out.
\end{proof}

\begin{lemma} \label{splitlem}
Let $H \subset G$ be an inclusion of topological groups. Then the pointed set $\map(X,G/H)$ carries a natural left action of the group $\map(X,G)$ and the functor
$$X \mapsto \frac{\map(X,G/H)}{\map(X,G)}$$
is split exact.
\end{lemma}
\begin{proof}
We need to show that the map

\begin{equation}\label{map:G}
\frac{\map(X,G/H)}{\map(X,G)} \to \frac{\map(X_1,G/H)}{\map(X_1,G)} \times_{\frac{\map(X_{12},G/H)}{\map(X_{12},G)}} \frac{\map(X_2,G/H)}{\map(X_2,G)}
\end{equation}
is a surjection with trivial kernel. Let $f \colon X \to G/H$ be
such that its pull-backs  $f_i:X_i\to G/H$ admit continuous lifts
$\hat{f}_i \colon X_i \to G$. Although the pull-backs of $\hat{f}_1$
and $\hat{f}_2$ to $X_{12}$ might not agree, we can fix this
problem. Let $\sigma$ be a continuous splitting of the inclusion
$X_{12} \hookrightarrow X_1$. Define a map \[ \gamma \colon X_1 \to
H,\quad \gamma(x)=(\hat{f}_1|_{X_{12}}(\sigma(x))^{-1} \cdot
(\hat{f}_2|_{X_{12}}(\sigma(x))
\]
Note $\hat{f}_1 \cdot \gamma$ is still a lift of $f_1$ and agrees with $\hat{f}_2$ on $X_{12}$; hence they define a map $\hat{f}:X\to G$ which lifts $f$. This proves that \eqref{map:G} has trivial kernel. Let now $f_1 \colon X_1 \to G/H$ and $f_2 \colon X_2 \to G/H$ be such that there exists a function $\theta \colon X_{12} \to G$ with $\theta(x) \cdot f_1(x) = f_2(x)$ for all $x \in X_{12}$. Using the splitting $\sigma$ of the inclusion $X_{12} \hookrightarrow X_1$ again, we can extend $\theta$ to $X_1$ to obtain
$f'_1(x) = \theta(\sigma(x)) f_1(x)$, for $x \in X_1$. Note $f'_1$ is just another representative of the class of $f_1$. Since $f'_1$ and $f_2$ agree on $X_{12}$, we conclude that there exists a continuous map $f \colon X \to G/H$, which pulls back to $f'_1$ on $X_1$ and to $f_2$ on $X_2$. This proves that \eqref{map:G} is surjective.
\end{proof}

\begin{proposition}\label{prop:homcor}
Let $\mathfrak{C}$ be either the category $\comp$ of compact Hausdorff spaces or the full subcategory $\pol$ of compact polyhedra.  Let $F:\mathfrak{C}\to \ab$ be a split-exact functor. Assume that $F(X)=0$ for contractible $X\in\mathfrak{C}$.
Then $F$ is homotopy invariant.
\end{proposition}
\begin{proof} We have to prove that if $X\in\mathfrak{C}$ and $1_X \times 0:X\to X\times [0,1]$ is the inclusion, then $F(1_X \times 0) \colon F(X \times [0,1]) \to F(X)$ is a bijection. Since it is obviously a split-surjection it remains to show that this map is injective. Consider the pushout diagram
\[
\xymatrix{X\ar[d]\ar[r]^(.4){1_X\times 0}&X\times [0,1]\ar[d]\\ \star \ar[r]& cX}
\]
By split-exactness, the map $F(cX)\to P:=F(*)\times_{F(X)}F(X\times [0,1])$ is onto. Since we are also assuming that $F$ vanishes on contractible spaces, we further have $F(*)=F(cX)=0$,
whence $P=\ker(F(1_X\times 0))=0$.
\end{proof}

\subsection{Algebraic $K$-theory.}\label{subsec:kth}

In the previous subsection we considered contravariant functors on
spaces; now we turn our attention  to the dual picture of covariant
functors from categories of algebras to pointed sets or abelian
groups. The most important example for us is algebraic $K$-theory.
Before we go on, we want to quickly recall some definitions and
results. Let $R$ be a unital ring. The abelian group $K_0(R)$ is
defined to be the Grothendieck group of the monoid of isomorphism
classes of finitely generated projective $R$-modules with direct sum
as addition. We define
$$K_n(R) = \pi_n(B\GL(R)^+),\star),\quad \forall n \geq 1,$$
where $X \mapsto X^+$ denotes Quillen's plus-construction \cite{qui}. Bass' Nil $K$-groups of a ring are defined as
\begin{equation} \label{ink}
NK_n(R) = \coker\left( K_n(R) \to K_n(R[t]) \right).
\end{equation}
The so-called fundamental theorem gives an isomorphism
\begin{equation}\label{bhs}
K_n(R[t,t^{-1}]) = K_n(R) \oplus K_{n-1}(R) \oplus NK_n(R) \oplus NK_n(R), \quad \forall n \geq 1
\end{equation}
which holds for all unital rings $R$. One can use this to define $K$-groups and Nil-groups in negative degrees. Indeed, if one puts
$$K_{n-1}(R) = \coker\left( K_n(R[t]) \oplus K_n(R[t^{-1}]) \to K_n(R[t,t^{-1}]) \right),$$
negative $K$-groups can be defined inductively. There is a functorial spectrum $K(R)$, such that
\begin{equation} \label{kspecdef}
K_n(R) = \pi_n K(R)\qquad (n\in\Z).
\end{equation}
This spectrum can be constructed in several equivalent ways (e.g.
see \cite{gers2}, \cite{pw1},\cite[\S 5]{pw2}, \cite[\S6]{tho},
\cite{wag}). Functors from the category of algebras to spectra and
their properties will be studied in more detail in the next
subsection.

\vspace{0.2cm}

A ring $R$ is called $K_n$-\emph{regular} if the map $K_n(R)\to K_n(R[t_1,\dots,t_m])$ is an isomorphism
for all $m$; it is called $K$-\emph{regular} if it is $K_n$-regular for all $n$. It is well-known that if $R$ is a regular noetherian ring then $R$ is $K$-regular and $K_nR=0$ for $n<0$. In particular this applies when $R$ is the coordinate ring of a smooth affine algebraic variety over a field.
We think of the Laurent polynomial ring $R[t_1,t_1^{-1},\dots, t_n,t_n^{-1}]$ as the group ring $R[\Z^n]$ and use  the fact that if the natural map $K_0(R) \to K_0(R[\Z^n])$ is an isomorphism for all $n \in \N$, then all negative algebraic $K$-groups and all (iterated ) nil-$K$-groups in negative degrees vanish. This can be proved with an easy induction argument.

\vspace{0.2cm}

\begin{remark} Iterating the nil-group construction, one obtains the following formula for the $K$-theory of the polynomial ring in $m$-variables
\begin{equation}\label{formu:kpol}
K_n(R[t_1,\dots,t_m])=\bigoplus_{p=0}^mN^pK_n(R)\otimes \bigwedge^p\Z^m
\end{equation}
Here $\bigwedge^p$ is the exterior power and $N^pK_n(R)$ denotes the iterated nil-group defined using the analogue of formula (\ref{ink}). Thus a ring $R$ is $K_n$-regular if and only if $N^pK_n(R)=0$ for all $p\ge 0$. In \cite{bass} Hyman Bass raised the question of whether the condition that $NK_n(R)=0$ is already sufficient for $K_n$-regularity. This question was settled in the negative in \cite[Thm. 4.1]{chwbass}, where an example of a commutative algebra $R$ of finite type over $\Q$ was given such that
$NK_0(R)=0$ but $N^2K_0(R)\ne 0$.
On the other hand it was proved \cite[Cor. 6.7]{chwnk} (see also \cite{gubelbq}) that if $R$ is of finite
type over a large field such as $\R$ or $\C$, then $NK_n(R)=0$ does imply that $R$ is $K_n$-regular. This is already sufficient for our purposes, since the rings this paper is concerned with are algebras over $\R$. For completeness let us remark further that if $R$ is any ring such that $NK_n(R)=0$ for \emph{all} $n$
then $R$ is $K$-regular, i.e. $K_n$-regular for all $n\in\Z$. As observed by Jim Davis in \cite[Cor. 3]{davis} this follows from Frank Quinn's theorem that
the Farrell-Jones conjecture is valid for the group $\Z^n$ (see also \cite[4.2]{chwnk}).
\end{remark}

\subsection{Homology theories and excision.}\label{subsec:func-n-H}
We consider functors and homology theories of associative, not necessarily unital algebras over
a fixed field $k$ of characteristic zero. In what follows, $\rC$ will denote either the category $\rA/k$ of associative $k$-algebras or the full subcategory
$\Co/k$ of commutative algebras. A {\it homology theory} on $\rC$ is a functor $E:\rC \to \Spt$ to the category
of spectra which preserves finite products up to homotopy. That is, $E(\prod_{i\in I}A_i)\to \prod_{i\in I}E(A_i)$
is a weak equivalence for finite $I$.
If $A\in\rC$ and $n\in\Z$, we write $E_n(A)=\pi_nE(A)$ for the $n$-th stable homotopy group.
Let $E$ be a homology theory and let
\begin{equation}\label{seq}
0\to A\to B\to C\to 0
\end{equation}
be an exact sequence (or extension) in $\rC$. We say that $E$ {\it satisfies excision} for \eqref{seq}, if
$E(A)\to E(B)\to E(C)$ is a homotopy fibration. The algebra $A$ is {\it $E$-excisive} if $E$ satisfies
excision on any extension \eqref{seq} with kernel $A$. If $\fC\subset\rC$ is a subcategory, and
$E$ satisfies excision for every sequence \eqref{seq} in $\fC$, then we say that $E$ satisfies excision
on $\fC$.

\begin{remark} If we have a functor $E$ which is only defined on the subcategory $\rC_1\subset \rC$ of  unital algebras and unital homomorphisms, and which preserves finite products up to homotopy, then we can extend it to all of $\rC$ by setting
\[
E(A)=\hofi(E(A^+_k)\to E(k))
\]
Here $A^+_k$ denotes the unitalization of $A$ as a $k$-algebra.
The restriction of the new functor $E$ to unital algebras is not the same as the old one, but it is homotopy
equivalent to it. Indeed, for $A$ unital, we have $A^+_k \cong A \oplus k$ as $k$-algebras. Since $E$ preserves finite products, this implies the claim.
In this article, whenever we encounter a homology theory defined
only on unital algebras, we shall implicitly consider it extended to non-unital algebras by the procedure just
explained. Similarly, if $F:\rC_1\to\ab$ is a functor to abelian groups which preserves finite products,
it extends to all of $\rC$ by
\[
F(A)=\ker(F(A^+_k)\to F(k))
\]
In particular this applies when $F=E_n$ is the homology functor associated to a homology theory as above.
\end{remark}

The main examples of homology theories we are interested in are $K$-theory, Hochschild homology and the various variants of cyclic homology.
A milestone in understanding excision in $K$-theory is the following result of Andrei Suslin and Mariusz Wodzicki, see \cite{sw0, sw1}.
\begin{theorem}[Suslin-Wodzicki]\label{thm:swex}
A $\Q$-algebra $R$ is $K$-excisive if and only if for the
$\Q$-algebra unitalization $R^+_\Q=R\oplus \Q$ we have
\[
\tor^{R^+_\Q}_n(\Q,R)=0
\]
for all $n\ge 0$.
\end{theorem}

For example it was shown in \cite[Thm. C]{sw1} that any ring satisfying a certain ``triple factorization property" is $K$-excisive; since any $C^*$-algebra has this property, (\cite[Prop. 10.2]{sw1}) we have

\begin{theorem} [Suslin-Wodzicki]\label{thm:swcstar}
$C^*$-algebras are $K$-excisive.
\end{theorem}

Excision for Hochschild and cyclic homology of $k$-algebras,
denoted respectively $HH(/k)$ and $HC(/k)$, has been studied in detail by Wodzicki in
\cite{MR997314}; as a particular case of his results, we cite the following:
\begin{theorem}[Wodzicki]\label{thm:wodex}
The following are equivalent for a $k$-algebra $A$.
\begin{enumerate}
\item $A$ is $HH(/k)$-excisive.
\item $A$ is $HC(/k)$-excisive.
\item $\tor_*^{A^+_k}(k,A)=0$.
\end{enumerate}
\end{theorem}
Note that it follows from \eqref{thm:swex} and \eqref{thm:wodex} that a $k$-algebra $A$ is $K$-excisive
if and only if it is $HH(/\Q)$-excisive.

\begin{remark}\label{rem:wodpnas} Wodzicki has proved (see \cite[Theorems ~1 and ~4]{wodpnas})
that if $A$ is a $C^*$-algebra then $A$ satisfies the conditions of Theorem \ref{thm:wodex} for any subfield
$k\subset \C$.
\end{remark}

\subsection{Milnor squares and excision.} \label{subsec:splitexab}

We now record some facts about Milnor squares of $k$-algebras and excision.

\begin{definition}
A square of $k$-algebras
\begin{equation}\label{milsq}
\xymatrix{ A \ar[r] \ar[d] & B \ar[d]^f \\ C \ar[r]^g & D }
\end{equation}
is said to be a \emph{Milnor square} if it is a pull-back square and either $f$ or $g$ is surjective. It is said to be \emph{split} if either $f$ or $g$ has a section.
\end{definition}

Let $F$ be a functor from $\rC$ to abelian groups and let
\begin{equation}\label{spseq}
\xymatrix{0\ar[r]&A\ar[r]&B\ar[r]&C\ar@/_/[l]\ar[r]&0}
\end{equation}
be a split extension in $\rC$. We say that $F$ is {\it split exact} on \eqref{spseq} if
\begin{equation*}
\xymatrix{0\ar[r]&F(A)\ar[r]&F(B)\ar[r]&F(C)\ar[r]&0}
\end{equation*}
is (split) exact. If $\fC\subset \rC$ is a subcategory, and $F$ is split exact on all split
exact sequences contained in $\fC$, then we say that $F$ is split exact on $\fC$.

\begin{lemma}\label{lem:Exci} Let $E:\rC\to \Spt$ be a homology theory and \eqref{milsq} a Milnor square. Assume that
$\ker(f)$ is $E$-excisive. Then $E$ maps \eqref{milsq} to a homotopy cartesian square.
\end{lemma}

\begin{lemma}\label{lem:fsplit}
Let $F:\rC\to \ab$ be a functor, $\fC\subset \rC$ a subcategory closed under kernels, and \eqref{milsq} a split Milnor square in $\fC$. Assume that $F$ is split exact on $\fC$. Then the sequence
\[
0\to F(A)\to F(B)\oplus F(C)\to F(D)\to 0
\]
is split exact.
\end{lemma}

\section{Real algebraic geometry and split exact functors.}\label{sec:real}

In this section, we recall several results from real algebraic geometry and prove a theorem on the behavior of weakly split exact
functors with respect to proper semi-algebraic surjections (see Theorem \ref{thm:main}). Recall that a semi-algebraic set is \emph{a priori}
a subset of $\R^n$ which is described as the solution set of a finite number of polynomial equalities and inequalities. A map between semi-algebraic sets is semi-algebraic if its graph is a semi-algebraic set.
For general background on semi-algebraic sets, consult \cite{basu}.

\subsection{General results about semi-algebraic sets.}\label{subsec:genreal}

Let us start with recalling the following two propositions.

\begin{proposition}[see Proposition $3.1$ in \cite{basu}] \label{prop:closure}
The closure of a semi-algebraic set is semi-algebraic.
\end{proposition}

\begin{proposition}[see Proposition $2.83$ in \cite{basu}] \label{prop:image} Let $S$ and $T$ be semi-algebraic sets, $S'\subset S$
 and $T'\subset T$ semi-algebraic subsets and  $f\colon S \to T$ a semi-algebraic map.
Then $f(S')$ and $f^{-1}(T')$ are semi-algebraic.
\end{proposition}
\comment{
\begin{proof}
Suppose that $T \subset \R^k$. Since $S$ and $f$ are semi-algebraic, $f(S)$ is defined by some formula $\Psi(y_1,\dots,y_k)$ in the language of ordered fields with coefficients in $\R$. By Theorem $2.74$ in \cite{basu}, there exists a quantifier free formula $\Phi(y_1,\dots,y_k)$
with the property that for all $(y_1,\dots,y_k) \in \R^k$ the formula $\Psi(y_1,\dots,y_k)$ holds if and only $\Phi(y_1,\dots,y_k)$ holds. This shows that $f(S)$ is semi-algebraic.
\end{proof}
}
Note that a semi-algebraic map does not need to be continuous. Moreover, within the class of continuous maps, there are surjective maps $f \colon S \to T$, for which the quotient topology induced by $S$ does not agree with the topology on $T$.
An easy example is the projection map from $\{(0,0)\} \cup \{(t,t^{-1}) \mid t>0\}$ to $[0,\infty)$. This motivates the following definition.

\begin{definition}
Let $S,T$ be semi-algebraic sets.
A continuous semi-algebraic surjection $f \colon S \to T$ is said to be \emph{topological}, if for every semi-algebraic map $g \colon T \to Q$ the composition $g \circ f$ is continuous if and only if $g$ is continuous.
\end{definition}

Gregory Brumfiel proved the following result, which says that (under certain conditions) semi-algebraic equivalence relations lead to good quotients.

\begin{theorem}[see Theorem 1.4 in \cite{brum}] \label{thm:brum}
Let $S$ be a semi-algebraic set and let $R\subset S \times S$ be a closed semi-algebraic equivalence relation. If $\pi_1\colon R \to S$ is proper, then there exists a semi-algebraic set $T$ and a topological semi-algebraic surjection $f\colon S \to T$ such that
$$R= \{(s_1,s_2) \in S \times S \mid f(s_1)=f(s_2)\}.$$
\end{theorem}

\begin{remark}
Note that the properness assumption in the previous theorem is automatically fulfilled if $S$ is compact. This is the case we are interested in.
\end{remark}

\begin{corollary} \label{cor:brum}
Let $S,S'$ and $T$ be compact semi-algebraic sets and $$f\colon T \to S \quad \mbox{and} \quad  f'\colon T \to S'$$ be continuous semi-algebraic maps. Then, the topological push-out
$S \cup_T S'$ carries a canonical semi-algebraic structure such that the natural maps
$\sigma\colon S \to S \cup_T S'$ and $\sigma' \colon S' \to S \cup_T S'$ are semi-algebraic.
\end{corollary}

For semi-algebraic sets, there is an intrinsic notion of connectedness, which is given by the following definition.

\begin{definition}
A semi-algebraic set $S \subset \R^k$ is said to be \emph{semi-algebraically connected} if it is not a non-trivial union of semi-algebraic subsets which are both open and closed in $S$.
\end{definition}

One of the first results on connectedness of semi-algebraic sets is the following theorem.

\begin{theorem}[see Theorem 5.20 in \cite{basu}]
Every semi-algebraic set $S$ is the disjoint union of a finite number of semi-algebraically connected semi-algebraic sets which are both open and closed in $S$.
\end{theorem}

Next we come to aspects of semi-algebraic sets and continuous semi-algebraic maps which differ drastically from the expected results for general continuous maps. In fact, there is a far-reaching generalization of Ehresmann's theorem about local triviality of submersions. Let us consider the following definition.

\begin{definition}
Let $S$ and $T$ be two semi-algebraic sets and $f\colon S \to T$ be a continuous semi-algebraic function. We say that $f$ is a \emph{semi-algebraically trivial fibration} if there exist a semi-algebraic set $F$ and a semi-algebraic homeomorphism
$$\theta \colon T \times F \to S$$
such that $f \circ \theta$ is the projection onto $T$.
\end{definition}

A seminal theorem is Hardt's triviality result, which says that away from a subset of $T$ of smaller dimension, every map $f \colon S \to T$ looks like a semi-algebraically trivial fibration.

\begin{theorem}[\cite{basu} or \S4 in \cite{hardt}]\label{thm:hardt}
Let $S$ and $T$ be two semi-algebraic sets and $f \colon S \to T$ a continuous semi-algebraic function. Then there exists a closed semi-algebraic subset $V \subset T$ with $\dim V < \dim T$, such that $f$ is a semi-algebraically trivial fibration over every semi-algebraic connected component of $T \setminus V$.
\end{theorem}

We shall also need the following result about semi-algebraic triangulations.

\begin{theorem}[Thm. $5.41$ in \cite{basu}] \label{thm:triangular}
Let $S \subset \R^k$ be a compact semi-algebraic set, and let
$S_1,\dots,S_q$ be semi-algebraic subsets. There exists a simplicial complex $K$ and a semi-algebraic homeomorphism $h \colon |K| \to S$ such that each $S_j$ is the union of images of open simplices of $K$.
\end{theorem}

\begin{remark}
In the preceding theorem, the case where the subsets $S_j$ are closed is of special interest. Indeed, if the subsets $S_j$ are closed, the theorem implies that the triangulation of $S$ induces triangulations of $S_j$ for each $j \in \{1,\dots,q\}$.
\end{remark}

The following proposition is an application of Theorems \ref{thm:hardt} and \ref{thm:triangular}

\begin{proposition}\label{prop:descmap}
Let $T\subset\R^m$ be a compact semi-algebraic subset, $S$ a
semi-algebraic set and $f:S\to T$ a continuous semi-algebraic map.
Then there exist a semi-algebraic triangulation of $T$ and a finite
sequence of closed subcomplexes
\[\varnothing=V_{r+1} \subset V_r \subset V_{r-1} \subset \cdots \subset V_1 \subset V_0=T\]
such that the following conditions are satisfied

\begin{itemize}
\item[(i)] For each $k=0,\dots,r$, we have $\dim V_{k+1}<\dim V_k$,
and the map
\[
f|_{f^{-1}(V_{k} \setminus V_{k+1})} \colon f^{-1}(V_{k} \setminus V_{k+1}) \to V_{k} \setminus V_{k+1}
 \]
 is a semi-algebraically trivial fibration over every semi-algebraic connected component.

\item[(ii)] Each simplex in the triangulation lies in some $V_k$ and has at most one face of codimension one which intersects $V_{k+1}$.
\end{itemize}
\end{proposition}
\begin{proof}
We set $n= \dim T$. By Theorem \ref{thm:hardt}, there exists a closed semi-algebraic subset $V_1 \subset T$ with $\dim V_1 < n$, such that $f$ is a semi-algebraic trivial fibration over every semi-algebraic connected component of $T \setminus V_1$. Consider now $f|_{f^{-1}(V_1)} \colon f^{-1}(V_1) \to V_1$ and proceed as before to find $V_2 \subset V_1$. By induction, we find a chain
$$\varnothing \subset V_r \subset V_{r-1} \subset \cdots \subset V_1 \subset V_0=T$$
such that $V_k \subset V_{k-1}$ is a closed semi-algebraic subset and
$$f|_{f^{-1}(V_{k-1} \setminus V_k)} \colon f^{-1}(V_{k-1} \setminus V_k) \to V_{k-1} \setminus V_k$$ is a semi-algebraically trivial fibration over every semi-algebraic connected component, for all $k \in \{1, \dots,r\}$.
Using Theorem \ref{thm:triangular}, we may now choose a semi-algebraic triangulation of $T$ such that the subsets $V_k$ are sub-complexes. Taking a barycentric subdivision, each simplex lies in $V_k$ for some $k \in \{0,\dots,r\}$ and has at most one face of codimension one which intersects the set $V_{k+1}$.
\end{proof}

\subsection{The theorem on split exact functors and proper maps.}\label{subsec:main}

The following result is our main technical result. It is key to the proofs of Theorems \ref{thm:fib} and Theorem \ref{thm:htpy}.

\begin{theorem}\label{thm:main}
Let $T$ be a compact semi-algebraic subset of $\R^k$.
Let $S$ be a semi-algebraic set and let $f\colon S \to T$ be a proper continuous semi-algebraic surjection.  Then, there exists a semi-algebraic triangulation of $T$ such that for every weakly split-exact contravariant functor $F \colon \pol \to \set_+$ and every simplex $\Delta^n$ in the
triangulation, we have
\[
\ker(F(\Delta^n)\to F(f^{-1}(\Delta^n)))=*
\]
\end{theorem}
\begin{proof}
Choose a triangulation of $T$ and a sequence of subcomplexes $V_k \subset T$ as in Proposition \ref{prop:descmap}. We shall show that $\ker(F(\Delta^n)\to F(f^{-1}(\Delta^n)))=*$ for each simplex in the chosen triangulation. The proof is by induction on the dimension of the simplex. The statement is clear for zero-dimensional simplices since $f$ is surjective. Let $\Delta^n$ be an $n$-dimensional simplex in the triangulation. By assumption, $f$ is a semi-algebraically trivial fibration over $\Delta^n \setminus \Delta^{n-1}$ for some face $\Delta^{n-1} \subset \Delta^n$. Hence, there exists a semi-algebraic set $K$ and a semi-algebraic homeomorphism $$\theta\colon (\Delta^n \setminus \Delta^{n-1}) \times K \to f^{-1}(\Delta^n \setminus \Delta^{n-1})$$ over $\Delta^n \setminus \Delta^{n-1}$. Consider the inclusion $f^{-1}(\Delta^{n-1}) \subset f^{-1}(\Delta^n)$. Since $f^{-1}(\Delta^{n-1})$ is an absolute neighborhood retract, there exists a compact neighborhood $N$ of $f^{-1}(\Delta^{n-1})$ in $f^{-1}(\Delta^n)$ which retracts onto $f^{-1}(\Delta^{n-1})$.
We claim that the set $f(N)$ contains some standard neighborhood $A$ of $\Delta^{n-1}$. Indeed, assume that $f(N)$ does not contain standard neighborhoods. Then there exists a sequence in the complement of $f(N)$ converging to $\Delta^{n-1}$. Lifting this sequence one can choose a convergent sequence in the complement of $N$ converging to $f^{-1}(\Delta^{n-1})$. This contradicts the fact that $N$ is a neighborhood and hence, there exists a standard compact neighborhood $A$ of $\Delta^{n-1}$ in $\Delta^n$ such that $f^{-1}(A) \subset N$.
Since $(\Delta^n \setminus \Delta^{n-1}) \times K \cong f^{-1}(\Delta^n \setminus \Delta^{n-1})$, any retraction of $\Delta^n$ onto $A$ yields a retraction of
$f^{-1}(\Delta^n)$ onto $f^{-1}(A)$. We have that $f^{-1}(A) \subset N$, and thus we can conclude that $f^{-1}(\Delta^{n-1})$ is a retract of $f^{-1}(\Delta^n)$.

By Corollary \ref{cor:brum}, the topological push-out
$$
\xymatrix{f^{-1}(\Delta^{n-1}) \ar[r] \ar[d] & f^{-1}(\Delta^n)\ar@{.>}[d]  \\
\Delta^{n-1}\ar@{.>}[r] & Z}
$$
carries a semi-algebraic structure. Moreover, by weak split-exactness
\begin{equation}\label{eq:ker*}
\ker(F(Z) \to F(\Delta^{n-1}) \times_{F(f^{-1}(\Delta^{n-1}))} F(f^{-1}(\Delta^n)))=*.
\end{equation}
Note that $f^{-1}(\Delta^n \setminus \Delta^{n-1}) \subset Z$ by definition of $Z$. We claim that the natural map $\sigma\colon Z \to \Delta^n$ is a semi-algebraic split surjection. Indeed, identify
$f^{-1}(\Delta^n \setminus \Delta^{n-1})= (\Delta^n \setminus \Delta^{n-1})\times K$, pick $k \in K$, and consider
\[
G_k = \overline{\{(d,(d,k)) \in \Delta^n \times Z \mid d \in \Delta^n \setminus \Delta^{n-1}\}}.
\]
Then $G_k$ is semi-algebraic by Proposition \ref{prop:closure},  and therefore defines a continuous semi-algebraic map $\rho_k\colon \Delta^n \to Z$ which splits $\sigma$.
Thus $F(\sigma):F(\Delta^n)\to F(Z)$ is injective, whence
\begin{equation}\label{eq:ker*2}
\ker(F(\Delta^n) \to F(\Delta^{n-1}) \times_{F(f^{-1}(\Delta^{n-1}))} F(f^{-1}(\Delta^n)))=*,
\end{equation}
by \eqref{eq:ker*}. But the kernel of $F(\Delta^{n-1})\to F(f^{-1}(\Delta^{n-1}))$ is trivial by induction, so the same must be true of $F(\Delta^n)\to F(f^{-1}(\Delta^n))$, by \eqref{eq:ker*2}.
\end{proof}

\comment{
\begin{remark}
Consider $Y=\{(x,y) \in \R^2 \mid xy=1\} \cup \{(0,0)\}$ and the projection $\pi_1\colon Y \to \R$ onto the first component. Then, $\pi_1$ is surjective but not proper. Since the conclusion of Theorem \ref{thm:main} fails for $\pi_1$, this shows that the assumption of properness is necessary.
\end{remark}
}

\section{Large semi-algebraic groups and the compact fibration theorem.}\label{sec:fib}

\subsection{Large semi-algebraic structures.}\label{subsec:largesets}

Recall a partially ordered set $\Lambda$ is \emph{filtered} if for any two $\lambda,\gamma\in\Lambda$
there exists a $\mu$ such that $\mu\ge \lambda$ and $\mu\ge \gamma$. We shall say that a filtered poset $\Lambda$ is \emph{archimedian} if there exists a monotone map $\phi:\N\to \Lambda$ from the ordered set of
natural numbers which is cofinal, i.e. it is such that for every $\lambda\in\Lambda$ there exists an $n\in\N$ such that $\phi(n)\ge \lambda$.
If $X$ is a set, we write $P(X)$ for the partially ordered set of all subsets of $X$, ordered by inclusion.

A \emph{large semi-algebraic structure}
%(\LSS)
on a set $X$ consists of
\begin{itemize}
\item[(i)] An archimedian filtered partially ordered set $\Lambda$.
\item[(ii)] A monotone map $X:\Lambda\to P(X)$, $\lambda\mapsto X_\lambda$, such that $X=\bigcup_\lambda X_\lambda$.
\item[(iii)] A compact semi-algebraic structure on each $X_\lambda$ such that if $\lambda\le\mu$ then
the inclusion $X_\lambda\subset X_\mu$ is semi-algebraic and continuous.
\end{itemize}
We think of large a semi-algebraic structure on $X$ as an exhaustive
filtration $\{X_\lambda\}$ by compact semi-algebraic sets. A
structure  $\{X_\gamma:\gamma\in\Gamma\}$ is \emph{finer} than a
structure $\{X_\lambda:\lambda\in\Lambda\}$ if for every
$\gamma\in\Gamma$ there exists a $\lambda\in \Lambda$ so that
$X_\gamma\subset X_\lambda$ and the inclusion is continuous and
semi-algebraic. Two structures are \emph{equivalent} if each of them
is finer than the other. A \emph{large semi-algebraic set} is a set
$X$ together with an equivalence class of semi-algebraic structures
on $X$. If $X$ is a large semi-algebraic set, then any large
semi-algebraic structure $\{X_\lambda\}$ in the equivalence class
defining $X$ is called a \emph{defining structure} for $X$. If
$X=(X,\Lambda)$ and $Y=(Y,\Gamma)$ are large semi-algebraic sets,
then a set map $f:X\to Y$ is called a \emph{morphism} if for every
$\lambda\in\Lambda$ there exists a $\gamma\in\Gamma$ such that
$f(X_\lambda)\subset Y_\gamma$, and such that the induced map
$f\colon X_\lambda\to Y_\gamma$ is semi-algebraic and continuous. We
write $\cV_{\infty}$ for the category of large semi-algebraic sets.

\begin{remark} If $f:X\to Y$ is a morphism of large semi-algebraic sets, then we may choose structures
$\{X_n\}$ and $\{Y_n\}$ indexed by $\N$, such that $f$ strictly preserves filtrations, i.e. $f(X_n)\subset Y_n$ for all $n$. However, if $X=Y$,
then there may not exist a structure $\{X_\lambda\}$ such that $f(X_\lambda)\subset X_\lambda$.
\end{remark}

\begin{remark}\label{rem:large-ind}
Consider the category $\cV_{s,b}$ of compact semi-algebraic sets with \emph{continuous} semi-algebraic mappings and its ind-category $\mathrm{ind}$-$\cV_{s,b}$. The objects are functors
$$T \colon (X_T, \leq) \to \cV_{s,b}$$
where $(X_T,\leq)$ is a filtered partially ordered set.
We set
$$\hom(T,S) = \lim_{d \in X_T} {\rm co} \!  \!  \lim_{e \in X_S} \hom_{\cV_{s,b}}(T(d),S(e)).$$

The category of large semi-algebraic sets is equivalent to the subcategory of those ind-objects whose
structure maps are all injective and whose index posets are archimedean. In particular, a filtering colimit
of an archimedian system of injective homomorphisms of large semi-algebraic sets is again a large semi-algebraic set.
\end{remark}

If $X$ and $Y$ are large semi-algebraic sets with structures $\{X_\lambda:\lambda\in\Lambda\}$
and $\{Y_\gamma:\gamma\in\Gamma\}$ then the cartesian product $X\times Y$ is again a large semi-algebraic
set, with structure $\{X_\lambda\times Y_\gamma:(\lambda,\gamma)\in \Lambda\times\Gamma\}$. A \emph{large semi-algebraic group} is a group object $G$ in $\cV_\infty$.  Thus $G$ is a group which is a large semi-algebraic set and each of the maps defining the multiplication, unit, and inverse, are homomorphisms in $\cV_\infty$.
We shall additionally assume that $G$ admits a structure $\{G_\lambda\}$ such that $G_\lambda^{-1}\subset G_\lambda$. This hypothesis, although not strictly necessary, is verified by all the examples we shall consider, and makes proofs technically simpler. We shall also need the notions of large semi-algebraic vectorspace and of large semi-algebraic ring, which are defined similarly.

\begin{example}\label{exa:large}
Any semi-algebraic set $S$ can be considered as a large semi-algebraic set, with the structure defined by its
compact semi-algebraic subsets, which is equivalent to the structure defined by any exhaustive filtration
of $S$ by compact semi-algebraic subsets. In particular this applies to any finite dimensional real vectorspace $V$;
moreover the vectorspace operations are semi-algebraic and continuous, so that $V$ is a (large) semi-algebraic vectorspace. Any linear map between finite dimensional vectorspaces is semi-algebraic
and continuous, whence it is a homomorphism of semi-algebraic vectorspaces. Moreover, the same is
true of any multilinear map $f:V_1\times\dots\times V_n\to V_{n+1}$ between finite dimensional vectorspaces.
\end{example}

\begin{definition}
Let $V$ be a real vectorspace of countable dimension. The \emph{fine} large semi-algebraic structure $\cF(V)$ is that given by all the compact semi-algebraic subsets of all the finite dimensional subspaces
of $V$.
\end{definition}

\begin{remark} The fine large semi-algebraic structure is reminiscent of the \emph{fine locally convex topology} which makes every complex algebra of countable dimension into a locally convex algebra. For details, see \cite{bourb}, chap. II, \S 2, Exercise 5.
\end{remark}

\begin{lemma}\label{lem:multilin}
Let $n\ge 1$, $V_1,\dots,V_{n+1}$ be  countable dimensional $\R$-vectorspaces, and  $f:V_1\times\dots\times V_n\to V_{n+1}$ a multilinear map. Equip each  $V_i$ with the fine large semi-algebraic structure and $V_1\times\dots\times V_n$ with the product large semi-algebraic structure. Then $f$ is a morphism of large
semi-algebraic sets.
\end{lemma}
\begin{proof} In view of the definition of the fine
large semi-algebraic structure, the general case is immediate from the finite dimensional case.
\end{proof}

\begin{proposition}\label{prop:glnlarge}
Let $A$ be a countable dimensional $\R$-algebra, equipped with the fine large semi-algebraic structure. Then:
\begin{itemize}
\item[(i)] $A$ is a large semi-algebraic ring.
\item[(ii)] Assume $A$ is unital. Then the group $\GL_n(A)$ together with the structure
$$\cF(\GL_n(A))=\{\{g\in \GL_n(A): g,g^{-1}\in F\}:F\in\cF(M_n(A))\}$$
is a large semi-algebraic group, and if $A\to B$ is an algebra homomorphism, then the induced group
homomorphism  $\GL_n(A)\to \GL_n(B)$ is a homomorphism of large semi-algebraic sets.
\end{itemize}
\end{proposition}
\begin{proof} Part (i) is immediate from Lemma \ref{lem:multilin}. If $F\in\cF(M_n(A))$, write
\begin{equation}\label{formu:glF}
\GL_n(A)^F=\{g\in \GL_n(A): g,g^{-1}\in F\}
\end{equation}
We will show that $\GL_n(A)^F$ is a compact semi-algebraic
set. Write $m,\pi_i:M_n(A)\times M_n(A)\to M_n(A)$ for the multiplication and projection maps $i=1,2$,
and $\tau:M_n(A)\times M_n(A)\to M_n(A)\times M_n(A)$ for the permutation of factors. If $A$ is unital, then
\[
\GL_n(A)^F=\pi_1(m_{|F\times F}^{-1}(1)\cap \tau m_{|F\times F}^{-1}(1))
\]
which is compact semi-algebraic, by Proposition \ref{prop:image}.
\end{proof}

We will also study the large semi-algebraic group $\GL_S(R)$ for a subset $S \subset \Z$. This is understood to be the group of matrices $g$ indexed by $\Z$, where $g_{i,j}=\delta_{i,j}$ if either $i$ or
 $j$ is not in $S$.

\begin{corollary}\label{cor:gl}
If $A$ is unital and $S\subset \Z$, then $\GL_S(A)$ carries a natural large semi-algebraic group structure,
namely that of the colimit $\GL_S(A)= \bigcup_T\GL_T(A)$ where $T$ runs among the finite subsets of $S$.
\end{corollary}
%\begin{proof} Write
%\end{proof}

\begin{remark}\label{rem:glnuni}
If $A$ is any, not necessarily unital ring, and $a,b\in M_nA$, then $a\star b=a+b+ab$ is an associative
operation, with neutral element the zero matrix; the group $\GL_n(A)$ is defined as the set of all matrices
which are invertible under $\star$.
If $A$ happens to be unital, the resulting group is isomorphic to that of invertible matrices via $g\mapsto
g+1$. If $A$ is any countable dimensional $\R$-algebra, then part ii) of Proposition \ref{prop:glnlarge} still holds if we replace $g^{-1}$ by the inverse of $g$ under the operation $\star$ in the definition
of $\GL_n(A)^F$. Corollary \ref{cor:gl} also remains valid in the non-unital case, and the proof is
the same.
\end{remark}

\subsection{Construction of quotients of large semi-algebraic groups.}\label{subsec:quot}

\begin{lemma}\label{lem:compastru}
Let $X\subset Y$ be an inclusion of large semi-algebraic sets. Then the following are equivalent.
\begin{itemize}
\item[(i)] There exists a defining structure $\{Y_\lambda\}$ of $Y$ such that $\{Y_\lambda\cap X\}$
is a defining structure for $X$.
\item[(ii)] For every defining structure $\{Y_\lambda\}$ of $Y$, $\{Y_\lambda\cap X\}$
is a defining structure for $X$.
\end{itemize}
\end{lemma}

\begin{definition}\label{defi:compa} Let $X\subset Y$ be an inclusion of large semi-algebraic sets. We say that $X$ is \emph{compatible} with $Y$ if the equivalent conditions of Lemma \ref{lem:compastru} are satisfied.
\end{definition}

\begin{proposition}\label{prop:quot} Let $H\subset G$ be a inclusion of large semi-algebraic groups, $\{G_\lambda\}$ a defining structure for $G$ and $\pi:G\to G/H$ the projection. Assume that the inclusion is compatible in the
sense of Definition \ref{defi:compa}. Then $(G/H)_\lambda=\pi(G_\lambda)$ is a large semi-algebraic structure, and the resulting large semi-algebraic set $G/H$ is the categorical quotient in $\cV_\infty$.
\end{proposition}

\begin{proof}
The map $G_\lambda\to (G/H)_\lambda$ is the set-theoretical quotient modulo the relation $G_\lambda\times G_\lambda\supset R_\lambda=\{(g_1,g_2):g_1^{-1}g_2\in H\}$. Let $\mu$ be such that the product map $m$ sends $G_\lambda\times G_\lambda$ into $G_\mu$; write $\rm{inv}:G\to G$ for the map $\mathrm{inv}(g)=g^{-1}$. Then
\[
R_\lambda=(m\circ(\mathrm{inv},id))^{-1}(H\cap G_\mu)
\]
Because $H\subset G$ is compatible, $H\cap G_\mu\subset G_\mu$ is closed and semi-algebraic, whence the
same is true of $R_\lambda$. By Theorem \ref{thm:brum}, $(G/H)_\lambda$ is semi-algebraic and
$G_\lambda\to (G/H)_\lambda$ is semi-algebraic and continuous. It follows that the $(G/H)_\lambda$ define
a large semi-algebraic structure on $G/H$ and that the projection is a morphism in $\cV_\infty$. The universal property of the quotient is straightforward.
\end{proof}

\begin{examples}\label{exas:pn}
Let $R$ be a unital, countable dimensional $\R$-algebra. The set
\[
P_n(R)=\{g(1_n\oplus 0_\infty)g^{-1}:g\in\GL(R)\}
\]
of all finite idempotent matrices which are conjugate to the $n\times n$ identity matrix can be written as a quotient of a compatible inclusion of large semi-algebraic groups. We have
\[
P_n(R)=\frac{\GL(R)_{[1,\infty)}}{\GL_{[1,n]}(R) \times \GL_{[n+1,\infty)}(R)}.
\]
On the other hand, since  $P_n(R)\subset M_\infty R$, it also carries another large semi-algebraic structure, induced by the fine structure on $M_\infty R$. Since $g\mapsto g(1_n\oplus 0_\infty)g^{-1}$ is
semi-algebraic, the universal property of the quotient \ref{prop:quot} implies that the quotient structure is finer than the subspace structure. Similarly, we may write the set of those idempotent matrices which are stably
conjugate to $1_n\oplus 0_\infty$ as
\begin{align*}
P_n^\infty (R)=&\{g(1_\infty\oplus 1_n\oplus 0_\infty)g^{-1}:g\in \GL_\Z(R)\}\\
              =& \frac{\GL_{(-\infty,+\infty)}(R)}{\GL_{(-\infty,n]}(R) \times \GL_{[n+1,\infty)}(R)}
\end{align*}
Again this carries two large semi-algebraic structures; the quotient structure, and that coming from
the inclusion
\[
P_n^\infty(R)\subset M_2(M_\infty(R)^+)
\]
into the $2\times 2$ matrices of the unitalization of $M_\infty(R)$.
\end{examples}

\subsection{Large semi-algebraic sets as compactly generated spaces.}\label{subsec:LSSop}

A topological space $X$ is said to be \emph{compactly generated} if it carries the inductive topology with respect to its compact subsets, i.e. a map $f \colon X \to Y$ is continuous if and only if its restriction to any compact subset of $X$ is continuous. In other words, a subset $U \subset X$ is open (resp.\ closed) if and only if $U \cap K$ is open (resp.\ closed) in $K$  for every $K \subset X$ compact. Observe that
any filtering colimit of compact spaces is compactly generated. In particular, if $X$ is a large semi-algebraic set, with defining structure $\{X_\lambda\}$, then $X=\bigcup_\lambda X_\lambda$ equipped with the colimit topology is compactly
generated, and this topology depends only on the equivalence class of the structure $\{X_\lambda\}$.
In what follows, whenever we regard a large semi-algebraic set as a topological space, we will implicitly
assume it equipped with the compactly generated topology just defined. Note further that any morphism of large semi-algebraic sets is continuous for the compactly generated topology. Lemma \ref{lem:compatop}
characterizes those inclusions of large semi-algebraic sets which are closed subspaces, and Lemma \ref{lem:g/h} concerns quotients of large semi-algebraic groups with the compactly generated
topology. Both lemmas are straightforward.

\begin{lemma}\label{lem:compatop}
An inclusion $X\subset Y$ of large semi-algebraic sets is compatible if and only if $X$ is a closed subspace of $Y$ with respect to the compactly generated topologies.
\end{lemma}

\begin{lemma}\label{lem:g/h}
Let $H\subset G$ be a compatible inclusion of large semi-algebraic groups. View $G$ and $H$ as topological
groups equipped with the compactly generated topologies. Then the compactly generated
topology associated to the quotient large semi-algebraic set $G/H$ is the quotient topology.
\end{lemma}

We shall be concerned with large semi-algebraic groups which are Hausdorff for the compactly generated
topology. The main examples are countable dimensional $\R$-vector spaces and groups such as  $\GL_S(A)$, for some subset $S \subset \Z$ and some countably dimensional unital $\R$-algebra $A$.

%\begin{lemma}\label{lem:glhaus}
%\begin{itemize}
%\item[(i)] Let $V$ be a countable dimensional $\R$-vectorspace. Regard $V$ as a topological space with the compactly generated topology associated to the fine \LSS. Then $V$ is a Hausdorff space.
%\item[(ii)] Let $A$ be a countable dimensional unital $\R$-algebra, and $S\subset \Z$. Then $\GL_S(A)$,
%with the compactly generated topology associated to the \LSS of \ref{cor:gl},
%is a Hausdorff space.
%\end{itemize}
%\end{lemma}
%\begin{proof} Upon choosing a basis, we have $V=\R^{(\N)}$, which is a $CW$-complex, and in particular,
%a Hausdorff space. This proves i). In particular, any countable dimensional algebra is Hausdorff. This
%applies to the algebra $M_S(A)$ of matrices supported in $S\times S$. In view of Remark \ref{rem:glnuni}
%we may regard $\GL_S(A)$ as a subset of $M_S(A)$; the structure of Proposition \ref{prop:glnlarge} is
%finer than the induced structure, and thus the corresponding topology is finer than the induced one.
%Since the induced topology is Hausdorff, the same is true of the finer one.
%\end{proof}

\begin{lemma}\label{lem:map}
Let $\F=\R$ or $\C$, $V$ a countable dimensional $\F$-vectorspace, $A$ a countable dimensional $\F$-algebra,
$S\subset \Z$ a subset, $X$ a compact Hausdorff topological space, and $C(X)=\map(X,\F)$. Equip $V$, $A$, and $\GL_S(A)$ with the compactly generated topologies. Then
the natural homomorphisms
\[
C(X)\otimes_\F V\to \map(X,V)\quad \mbox{and} \quad \GL_S(\map(X,A))\to \map(X,\GL_S(A))
\]
are bijective.
\end{lemma}
\begin{proof}
It is clear that the both homomorphisms are injective. The image of the first one consists of those
continuous maps whose image is contained in a finitely generated subspace of $V$. But since $X$ is compact and $V$
has the inductive topology of all closed balls in a finitely generated subspace, every continuous map is
of that form. Next note that $\GL_S(A)=\bigcup_{S',F} \GL_{S'}(A)^F$ where the union runs among the finite subsets
of $S$ and the compact semi-algebraic sets of the form $F=M_{S'}(B)$ where $B$ is a compact
semi-algebraic subset of some finitely generated subspace of $A$. Hence any continuous map $f:X\to\GL_S(A)$ sends $X$ into some $M_{S' }B$, and
thus each of the entries $f(x)_{i,j}$ $i,j\in S'$ is a continuous function. Thus $f(x)$ comes from an element of $GL_S(\map(X,A))$.
\end{proof}

\subsection{The compact fibration theorem for quotients of large semi-algebraic groups.}\label{subsec:fib}

Recall that a continuous map $f:X\to Y$ of topological spaces is said to have the \emph{homotopy lifting property} with respect to a space $Z$ if for any solid arrow diagram
\[
\xymatrix{Z\ar[d]_{id\times 0}\ar[r]&X\ar[d]^f\\
Z\times [0,1]\ar[r]\ar@{.>}[ur]& Y}
\]
of continuous maps a continuous dotted arrow exists and makes both triangles commute. The map $f$ is
a (Hurewicz) fibration if it has the $HLP$ with respect to any space $Z$, and is a Serre fibration if
it has the $HLP$ with respect to all disks $D^n$ $n\ge 0$.

\begin{definition} Let $X,Y$ be topological space and $f \colon X \to Y$ a continuous map. We say that $f$ is a  \emph{compact fibration} if for every compact subspace $K$, the map $f^{-1}(K)\to K$ is a fibration.
\end{definition}

\begin{remark}
Note that every compact fibration is a Serre fibration. Also, since every map $p:E\to B$ with compact $B$
which is a locally trivial bundle is a fibration, any map $f:X\to Y$  such that the restriction  $f^{-1}(K)\to K$ to any compact subspace $K\subset Y$ is a locally trivial bundle, is a compact fibration. The notion of compact fibration comes up naturally in the study of homogenous spaces of infinite dimensional topological groups.
\end{remark}

\begin{theorem} \label{thm:fib}
Let $H \subset G$ be a compatible inclusion of large semi-algebraic groups. Then the quotient map $\pi\colon G \to G/H$ is a compact fibration.
\end{theorem}
\begin{proof}
Choose defining structures $\{H_p\}$ and $\{G_p\}$ indexed over $\N$ and such that $H_p=G_p\cap H$; let $\{(G/H)_p\}$ be as in Proposition \ref{prop:quot}. Since any compact
subspace $K\subset G/H$ is contained in some $(G/H)_p$, it suffices to show that the projection $\pi_p=\pi|_{\pi^{-1}((G/H)_p)} \colon \pi^{-1}((G/H)_p)\to (G/H)_p$ is a locally
trivial bundle. By a well-known argument (see e.g. \cite[Thm. 4.13]{switzer}) if the quotient map of a group by a closed subgroup admits local sections then it is a locally trivial bundle;
the same argument applies in our case to show that if $\pi_p$ admits local sections then it is a locally trivial bundle. Consider the functor
\[
F:\comp\to\set_+,\quad F(X)= \frac{\map(X,G/H)}{\map(X,G)}
\]
By Lemma \ref{splitlem}, $F$ is split exact. By Theorem \ref{thm:main} applied to $F$ and to the proper semi-algebraic surjection $G_p\to (G/H)_p$, there is a triangulation of $(G/H)_p$ such that
\begin{equation}\label{eq:ker}
\ker(F(\Delta^n)\to F(\pi_p^{-1}(\Delta^n)))=*
\end{equation}
for each simplex $\Delta^n$ in the triangulation. The diagram
\[
\xymatrix{
\pi^{-1}(\Delta^n)\cap G_p\ar[r]\ar[d]&G_p\ar[r]\ar[d]_{\pi_p}& G\ar[d]^{\pi}\\ \Delta^n\ar[r]&(G/H)_p\ar[r]& G/H}
\]
shows that the class of the inclusion $\Delta^n\subset G/H$ is an element of that kernel,
and therefore it can be lifted to a continuous map $\Delta^n\to G$, by \eqref{eq:ker}. Thus $\pi_p$ admits
a continuous section over every simplex in the triangulation. Therefore, using split-exactness of $F$, it admits a continuous section over each
of the open stars $\st^{o}(x)$ of the vertices of the barycentric subdivision. As the open stars of vertices form an open covering
 of $(G/H)_p$, we conclude that $\pi_p$ admits local sections. This finishes the proof.
\end{proof}

\section{Algebraic compactness, bounded sequences and algebraic approximation.}\label{sec:algcomp}

\subsection{Algebraic compactness.}\label{subsec:algcomp}

Let $R$ be a countable dimensional unital $\R$-algebra, equipped with the fine large semi-algebraic structure. Consider the large semi-algebraic sets
\begin{align}
M_\infty R\supset P_n(R) &=\frac{\GL(R)_{[1,\infty)}}{\GL_{[1,n]}(R) \times \GL_{[n+1,\infty)}(R)}\label{formu:pn}\\
M_2((M_\infty R)^+)\supset P_n^\infty (R)=& \frac{\GL_{(-\infty,+\infty)}(R)}{\GL_{(-\infty,n]}(R) \times \GL_{[n+1,\infty)}(R)}\label{formu:pinfty}
\end{align}
introduced in \ref{exas:pn}. Recall that each of $P_n(R)$, $P_n(R)^\infty$ carries two large semi-algebra structures:
the homogenous ones, coming from the quotients, and those induced by the inclusions above. As the homogeneous structures are finer than the induced ones, the same is true of the corresponding compactly generated topologies; they agree if and only if the corresponding large semi-algebraic structures are equivalent, or, what is the same, if every subset which is compact in the homogeneous topology is also compact in the induced one. This motivates the following definition.

\begin{definition}\label{defi:algcomp}
Let $R$ be a countable dimensional unital $\R$-algebra, equipped with the fine large semi-algebraic structure. We say that $R$ has the \emph{algebraic compactness} property if for every $n\ge 1$ the homogeneous and the induced large semi-algebraic structure of $P^\infty_n(R)$ agree.
\end{definition}

We show in Proposition \ref{prop:pn} below that if $R$ satisfies algebraic compactness, then the two topologies in \eqref{formu:pn}
also agree. For this we need some properties of compactly generated topological groups. All topological groups under consideration are assumed to be Hausdorff.

\begin{lemma} \label{lem:techlem}
Let $H$ and $H'$ be closed subgroups of a Hausdorff compactly generated group $G$. Then the quotient
topology on $H/H \cap H'$ is the subspace topology inherited from the quotient topology on $G/H'$.
\end{lemma}
\begin{proof}
First of all, it is clear that the canonical inclusion map $\iota \colon H/H \cap H' \to G/H'$ is continuous. Indeed, let $\pi:G\to G/H'$ be the projection; identify $H\to H/H\cap H'$ with the restriction of $\pi$. A subset
$A\subset G/H'$ is closed if and only if $\pi^{-1}(A) \cap K$ is closed for $K \subset G$ compact. Hence, $\pi^{-1}(A) \cap H \cap K = \pi^{-1}(A \cap \pi(H)) \cap K$ is closed and the claim follows since compact subsets of $H$ are also compact in $G$. Let now $A \subset H/H \cap H'$ be closed, i.e.\ $\pi^{-1}(A) \cap K'$ closed for every compact $K' \subset H$. For compact $K \subset G$, the set $K'=K \cap H$ is compact in $H$ and we get that $\pi^{-1}(A) \cap K$ closed in $H$ and hence in $G$. This finishes the proof.
\end{proof}

\begin{proposition}\label{prop:pn}
Let $R$ be a unital, countable dimensional $\R$-algebra, equipped with the fine large semi-algebraic structure. Assume that $R$ has the algebraic compactness property. Then the homogeneous and induced large semi-algebraic structures of $P_n(R)$ agree.
\end{proposition}
\begin{proof}
Consider the diagram
\[
\xymatrix{
\GL_{(-\infty,n]}(R)  \times \GL_{[n+1,\infty)}(R) \ar[r] & \GL_{(-\infty,\infty)}(R) \ar[r] & { P}^{\infty}_n(R) \\
\GL_{[1,n]}(R) \times \GL_{[n+1,\infty)}(R) \ar[u] \ar[r]  & \GL_{[1,\infty)}(R) \ar[r] \ar[u]& { P}_n(R). \ar[u]
}
\]
Now apply Lemma \ref{lem:techlem}.
\end{proof}

\subsection{Bounded sequences, algebraic compactness and $K_0$-triviality.}
%\label{subsec:bndalgcomp}

Let $X$ be a large semi-algebraic set and let $\{X_\lambda\}$ be a defining structure. The space of \emph{bounded sequences} in $X$ is
\[
\ell^\infty(X)=\ell^\infty(\N,X)=\{z\colon\N\to X \mid \exists \lambda :  z(\N)\subset X_\lambda\}
\]
Note that with our definition, the objects $\ell^{\infty}(\R)$ and $\ell^{\infty}{\C}$ coincide with the well-known spaces of bounded sequences.

\begin{lemma}\label{lem:bounded}
Let $\F=\R$ or $\C$ and $V$ a countable dimensional $\F$-vectorspace equipped with the fine large semi-algebraic structure. Then the natural map
\[
\ell^\infty(\F)\otimes_\F V\to \ell^\infty(V)
\]
is an isomorphism.
\end{lemma}
\begin{proof} Choose a basis $\{v_q\}$ of $V$. Any element of $\ell^\infty(\F)\otimes_\F V$ can be written uniquely as a finite sum $\sum_q \lambda_q\otimes v_q$; this gets mapped to the sequence $\{n \mapsto \sum_q\lambda_q(n) \cdot v_q\}$, which vanishes if and only if all the $\lambda_q$ are zero. This proves the injectivity statement. Let $z\in\ell^\infty(V)$; by definition, there
is a finite dimensional subspace $W\subset V$ and a bounded closed semi-algebraic subset $S\subset W$
such that $z(\N)\subset S$. We may assume that $S$ is a closed ball centered at zero, and that
$W$ is the smallest subspace containing $z(\N)$. Hence there exist $i_1<\dots<i_p\in\N$ such that
$B=\{z_{i_1},\dots,z_{i_p}\}$ is a basis of $W$. The map $W\to \R^p$, $w\mapsto [w]_B$ that sends a vector $w$ to the
$p$-tuple of its coordinates with respect to $B$ is linear and therefore bounded. In particular
there exists a $C>0$ such that $||[w]_B||_\infty<C$ for all $w\in S$. Thus we may write
$z=\sum_{j=1}^p \lambda_i z_{i_j}$ with $\lambda_i\in \ell^\infty(\F)$. This proves the surjectivity
assertion of the lemma.
\end{proof}

\begin{proposition}\label{prop:algcompeq}
Let $R$ be a unital, countable dimensional $\R$-algebra, equipped with the fine large semi-algebraic structure. Then the following are equivalent
\begin{itemize}
\item[(i)] $R$ has the algebraic compactness property.
\item[(ii)] For every $n$ the map
\begin{equation}\label{map:algcomp}
P^\infty_n(\ell^\infty(R))\to \ell^\infty( P^\infty_n(R))
\end{equation}
is surjective.
\item[(iii)] The map $K_0(\ell^\infty(R))\to \prod_{r\ge 1}K_0(R)$ is injective.
\end{itemize}
\end{proposition}
\begin{proof} Choose countable indexed structures $\{G_n\}$ on $G=\GL_\Z(R)$, and $\{X_n\}$ on
$X=M_2((M_\infty R)^+)$. Let
$\pi:G\to G/H=P_n^\infty(R)$ be the projection. We know already (see \ref{exas:pn}) that the induced structure on $P^\infty_n(R)$ is coarser than the homogeneous one,  i.e. that each $(G/H)_r=\pi(G_r)$ is contained in some $X_m$. Assertion (i) is therefore equivalent to saying that each $P^\infty_n(R)\cap X_m$ is contained in some $\pi(G_r)$. Negating this we obtain a bounded sequence $e=(e_r)$ of idempotent matrices, i.e.\ $e \in \ell^{\infty}(P_n^{\infty}(R))$ with respect to the induced large semi-algebraic structure on $P_n^{\infty}(R)$,
each $e_r$ is equivalent to $1_{\infty} \oplus (1_n\oplus 0_\infty)$ in $M_2(M_{\infty}(R)^+)$, but there is no sequence $(g_r)$ of invertible matrices in $GL_2(M_{\infty}(R)^+)$ such that
$g_re_rg_r^{-1} = 1_{\infty} \oplus (1_n\oplus 0_\infty)$ and both $(g_r)$ and $(g_r^{-1})$ are bounded. In other words, $e$ is not in $P_n^\infty(\ell^\infty(R))$. We have shown that the negation of (i) is equivalent to that of (ii). Next
note that every element $x\in K_0(\ell^\infty(R))$ can be written as a difference $x=[e]-[1_\infty \oplus 0_\infty]$ with $e \in M_2(M_{\infty}(R)^+)$ idempotent and $e \equiv 1_{\infty} \oplus 0_{\infty}$ modulo the ideal $M_2M_{\infty}(R)$. The idempotent $e$ is determined up to conjugation by $GL_2(M_{\infty}(R)^+)$.
The element $x$ goes to zero in $\prod_{r\ge 1}K_0(R)$ if and only if each $e_r$ is conjugate to $1_{\infty} \oplus 0_\infty$. Hence condition (iii) is satisfied if (ii) is. The converse follows easily. Indeed, for any sequence $(e_r)$ as above, we see that the image of the classes $[e]-[1_\infty \oplus 0_\infty]$ and $[1_{\infty} \oplus (1_n \oplus 0_{\infty})] - [1_{\infty} \oplus 0_{\infty}]$ in $\prod_{p\ge 1}K_0(R)$ coincide. Hence, by injectivity of the comparison map, $e$ is conjugate to $1_{\infty} \oplus (1_n \oplus 0_{\infty})$ and we get a sequence of invertible elements $(g_r)$ in $GL_2(M_{\infty}(R)^+)$ such that $g_r$ conjugates $e_r$ to $1_{\infty} \oplus (1_n \oplus 0_{\infty})$ and the sequences $(g_r)$ and $(g_r^{-1})$ are bounded. This completes the proof.
\end{proof}

\begin{example}\label{exa:calgcomp}
Both $\R$ and $\C$ have the algebraic compactness property since the third condition is well-known to be satisfied. Indeed, $\ell^{\infty}(\R)$ and $\ell^{\infty}(\C)$ are (real) $C^*$-algebras, and one can easily compute that
$$ K_0(\ell^{\infty}(\C)) = \ell^{\infty}(\Z) \subset \prod_{n \geq 1} \Z  = \prod_{n \geq 1} K_0(\C).$$
The same computation applies to $\R$ in place of $\C$.
\end{example}

\subsection{Algebraic approximation and bounded sequences.}\label{subsec:approx}

\begin{theorem}\label{thm:boundedapprox}
Let $F$ and $G$ be functors from commutative $\C$-algebras to sets.
Assume that both $F$ and $G$ preserve filtering colimits. Let $\tau:F\to G$ be a natural transformation.
Assume that $\tau(\cO(V))$ is injective (resp. surjective) for each smooth affine algebraic variety $V$ over $\C$.
Then $\tau(\ell^\infty(\C))$ is injective (resp. surjective).
\end{theorem}
\begin{proof}
Let $\cF \subset \ell^{\infty} \C$ be a finite subset. Put $A_\cF = \C\langle  \cF \rangle \subset
\ell^{\infty}(\C)$ for the unital subalgebra generated by $\cF$. Because $A_\cF$ is reduced, it corresponds to an affine algebraic variety $V_\cF$ and the inclusion $A_\cF \subset \ell^{\infty} \C$ is dual to a map with pre-compact image $\iota_\cF \colon \N \to (V_\cF)_{an}$ to the analytic variety associated to $V_\cF$. Thus we may
write $\ell^\infty(\C)$ as the filtering colimit
\begin{equation}\label{eq:bapro}
\ell^\infty(\C)=\colim_{\N\to V_{an}}\cO(V)
\end{equation}
Here the colimit is taken over all maps $\iota:\N\to V_{an}$ whose codomain is the associated analytic variety of the closed points of some affine algebraic variety $V$, and which have precompact image in the euclidean topology. We claim that
every such map factors through a map $V'_{an}\to V_{an}$ with $V'$ smooth and affine. Note that the claim
implies that we may write \eqref{eq:bapro} as a colimit of smooth algebras; the theorem is immediate from this.
Recall Hironaka's desingularization (see \cite{hironaka})
provides a proper surjective homomorphism of algebraic varieties $\pi:\tilde V \to V$
from a smooth quasi-projective variety. Thus the induced map $\pi_{an}:\tilde{V}_{an}\to V_{an}$ between
the associated analytic varieties is proper and surjective for the usual euclidean topologies.
It follows from this that we can lift $\iota$ along $\pi_{an}$. Next, Jouanoulou's
device (see \cite{jou}) provides a smooth affine vector bundle torsor $\sigma \colon V' \to \tilde V$; the associated
map $\sigma_{an}$ is also a bundle torsor, and in particular a fibration and weak equivalence.
Because $\tilde{V}_{an}$ is a $CW$-complex, $\sigma_{an}$ admits a continuous section.
Thus $\iota_F$ finally factors through the smooth affine variety $V'_\cF$.
\end{proof}

\begin{remark}\label{rem:surreal}
The proof above does not work in the real case, since a desingularization $\tilde{V}\to V$ of real algebraic
varieties need not induce a surjective map between the corresponding real analytic (or semi-algebraic) varieties. For example,
consider $R=\R[x,y]/\langle x^2+y^2-x^3\rangle$. The homomorphism $f:R\to \R[t]$, $f(x,y)\mapsto f(t^2+1,t(t^2+1))$, is
injective and $\R[t]$ is integral over $R$. Thus the induced scheme homomorphism $f_{\#}:\mathbb{A}^1_\R=\Spec\, \R[t]\to V=\Spec \,\R[x,y]/\langle x^2+y^2-x^3\rangle$
is a desingularization; it is finite (whence proper) and surjective, and an isomorphism outside of the point zero, represented by the maximal
ideal $\frakm=\langle x,y\rangle\in V$. But note that the pre-image of $\frakm$ consists just of the maximal ideal $\langle t^2+1\rangle$, which has residue field $\C$; this means that the pre-image of zero has no real points. Therefore the restriction of $f_\#$ to real points is not surjective.
\end{remark}
\subsection{The algebraic compactness theorem.}
%\label{subsec:algcomp2}

\begin{theorem}\label{thm:algcomp}
Let $R$ be a countable dimensional unital $\C$-algebra. Assume that the map
$K_0(\cO(V))\to K_0(\cO(V)\otimes R)$ is an isomorphism for every affine smooth algebraic variety $V$ over $\C$.
Then $R$ has the algebraic compactness property.
\end{theorem}
\begin{proof} We have a commutative diagram
\[
\xymatrix{K_0(\ell^\infty(R))\ar[r]& \prod_{p\ge 1}K_0(R)\\
          K_0(\ell^\infty(\C))\ar[u]\ar[r]& \prod_{p\ge 1}K_0(\C)\ar[u]}
\]
The bottom row is a monomorphism by Example \ref{exa:calgcomp}. Our hypothesis on $R$ together with Theorem \ref{thm:boundedapprox} applied to the natural transformation $K_0(-)\to K_0(-\otimes_\C R)$ imply that both columns are isomorphisms. It follows that the top row is injective, which by Proposition \ref{prop:algcompeq} says that $R$ satisfies algebraic compactness.
\end{proof}

\section{Applications: projective modules, lower $K$-theory, and bundle theory.}\label{sec:pgubel}

\subsection{Parametrized Gubeladze's theorem and Rosenberg's conjecture.}\label{subsec:pgubel}

All monoids considered are commutative, cancellative and without nonzero torsion
elements. If $M$ is cacellative then it embeds into its total quotient group $G(M)$. A cancellative monoid $M$ is said to be seminormal
if for every element $x$ of the total quotient group $G(M)$ for which $2x$ and $3x$
are contained in the monoid $M$, it follows that $x$ is contained in the monoid $M$.

\vspace{0.2cm}

The following is a particular case of a theorem of Joseph Gubeladze, which in turn generalized the
celebrated theorem of Daniel Quillen \cite{quiqs} and Andrei Suslin \cite{susqs} which settled Serre's conjecture: every finitely generated
projective module over a polynomial ring over a field is free.

\begin{theorem}[see \cite{gubelan1},\cite{gubelan2}] \label{thm:gubel}
Let $D$ be a principal ideal domain, and $M$ a commutative, cancellative, torsion free,
seminormal monoid. Then every finitely generated projective module over the monoid
algebra $D[M]$ is free.
\end{theorem}

We shall also need the following generalization of Gubeladze's theorem, due to Richard G. Swan. Recall that
if $R\to S$ is a homomorphism of unital rings, and $M$ is an $S$-module, then we say that $M$ is \emph{extended}
from $R$ if there exists an $R$-module $N$ such that $M\cong S\otimes_RN$ as $S$-modules.

\begin{theorem}[see \cite{swanpol}] \label{thm:swan}
Let $R=\cO(V)$ be the coordinate ring of a smooth affine algebraic variety over a field, and let $d=\dim V$.
Also let $M$ be a torsion-free, seminormal, cancellative monoid. Then all finitely generated projective
$R[M]$-modules of rank $n>d$ are extended from $R$.
\end{theorem}

In the next theorem and elsewhere below, we shall consider only the complex case; thus in what follows, $C(X)$ shall always means $\map(X,\C)$.

\begin{theorem}\label{thm:pswan}
Let $X$ be a contractible compact space, and $M$ an abelian, countable, torsion-free, seminormal,
cancellative monoid. Then
every finitely generated projective module over $C(X)[M]$ is free.
\end{theorem}
\begin{proof}
The assertion of the theorem is equivalent to the assertion that every idempotent matrix with coefficients
in $C(X)[M]$ is conjugate to a diagonal matrix with only zeroes and ones in the diagonal. By Lemma \ref{lem:map}, an idempotent matrix with coefficients in $C(X)[M]$
is the same as a continuous map from $X$ to the space $\Idem_\infty(\C[M])$ of all idempotent matrices in
$M_\infty(\C[M])$, equipped with the induced topology. Now observe that, since the trace map $M_\infty \C[M]\to \C[M]$ is continuous,
so is the rank map $\Idem_\infty(\C[M])\to \N_{0}$. Hence by Theorem \ref{thm:gubel}, the space $\Idem_\infty(\C[M])$ is the topological coproduct
$\Idem_\infty(\C[M])=\coprod_nP_n(\C[M])$, and thus any continuous map $e:X\to \Idem_\infty(\C[M])$
factors through a map $e:X\to P_n(\C[M])$. By Theorem \ref{thm:swan} and Theorem \ref{thm:algcomp},
the induced topology of $P_n(\C[M])=\GL(\C[M])/\GL_{[1,n]}(\C[M])\times \GL_{[n+1,\infty)}(\C[M])$ coincides with
the quotient topology. By Theorem \ref{thm:fib}, $e$ lifts to a continuous map $g:X\to \GL(\C[M])$. By
Lemma \ref{lem:map}, $g\in \GL(C(X)[M])$ and conjugates $e$ to $1_n\oplus 0_\infty$. This concludes the proof.
\end{proof}

\begin{theorem}\label{thm:k0htpy}
The functor $\comp\to \ab$, $X\mapsto K_0(C(X)[M])$
is homotopy invariant.
\end{theorem}
\begin{proof}
Immediate from Theorem \ref{thm:pswan} and Proposition \ref{prop:homcor}.
\end{proof}

\begin{theorem}[Rosenberg's conjecture] \label{thm:rosen1}
The functor $\comp\to \ab$, $X\mapsto K_{-n}(C(X))$ is homotopy invariant for $n>0$.
\end{theorem}
\begin{proof} By \eqref{bhs}, $K_{-n}(C(X))$ is naturally a direct summand of $K_0(C(X)[\Z^n])$, whence
it is homotopy invariant by Theorem \ref{thm:k0htpy}.
\end{proof}

\begin{remark}
Let $X$ be a compact topological space, $S^1$ the circle and $j\ge 0$. By \eqref{bhs}, Theorem \ref{thm:rosen1} and
excision, we have
\begin{align*}
K_{-j}(C(X\times S^1))=& K_{-j}(C(X))\oplus K_{-j-1}(C(X))\\
                      =& K_{-j}(C(X)[t,t^{-1}])
\end{align*}
Thus the effect on negative $K$-theory of the cartesian product of the maximal ideal spectrum $X=\Max(C(X))$ with $S^1$ is the same as that of taking the product
of the prime ideal spectrum $\spec C(X)$ with the algebraic circle $\spec(\C[t,t^{-1}])$. More generally, for the $C^*$-algebra tensor product $\otimes_{\rm min}$ and any commutative $C^*$-algebra $A$, we have $K_{-j}(A\otimes_{\rm min}C(S^1))=K_{-j}(A\otimes_\C\C[t,t^{-1}])$ $(j> 0)$.
\end{remark}

\begin{remark} \label{credithom}
Theorem \ref{thm:rosen1} was stated by Jonathan Rosenberg in \cite[Thm.\ 2.4]{roskk}
and again in \cite[Thm. 2.3]{rosop} for the real case.
Later, in \cite{roshan}, Rosenberg acknowledges that the proof was faulty, but conjectures the statement to be true. Indeed, a mistake was pointed out by Mark E. Walker (see line 8 on page 799 in \cite{fw} or line 12 on page 26 in \cite{roshan}).
In their work on semi-topological $K$-theory, Eric Friedlander and Mark E. Walker prove \cite[Thm.\ 5.1]{fw} that the negative algebraic $K$-theory of the ring $C(\Delta^n)$ of complex-valued continuous functions on the simplex vanishes for all $n$. We show in Subsection \ref{subsec:second} how another proof of Rosenberg's conjecture
can be obtained using the Friedlander-Walker result.
\end{remark}

\begin{remark}
The proof of Theorem \ref{thm:rosen1} does not need the detour of the proof of our main results in the case $n=1$.
Indeed, the ring of germs of continuous functions at a point in $X$ is a Hensel local ring with residue field $\C$ and Vladimir  Drinfel'd proves that $K_{-1}$ vanishes for Hensel local rings with residue field $\C$, see \cite[Thm.\ 3.7]{MR2181808}. This solves the problem locally and reduces the remaining complications to bundle theory. (This was observed by the second author in discussions with Charles Weibel at Institut Henri Poincar\'e, Paris in 2004.) No direct approach like this is known for $K_{-2}$ or in lower dimensions.
\end{remark}

Already in \cite{roskk}, Rosenberg computed the values of negative algebraic $K$-theory on commutative unital $C^*$-algebras, assuming the homotopy invariance result.

\begin{corollary}[Rosenberg, \cite{roshan}]\label{coro:bu}
Let $X$ be a compact topological space. Let ${ \rm \bf bu}$ denote the connective $K$-theory spectrum. Then,
\begin{gather*}
K_{-i}(C(X)) = {\rm bu}^{i}(X) = [\Sigma^iX,{\rm \bf bu}], \ \ i\ge 0.
\end{gather*}
\end{corollary}
The fact that connective $K$-theory shows up in this context was further explored and clarified in the thesis of the second author \cite{thomthesis}, which was also partially built on the validity of Theorem \ref{thm:rosen1}.

\subsection{Application to bundle theory: local triviality.}\label{subsec:loctriv}

Let $R$ be a countable dimensional $\R$-algebra. Any finitely generated $R$-module $M$ is a countable dimensional
vectorspace, and thus it can be regarded as a compactly generated topological space.
We consider not necessarily locally trivial bundles of finitely generated free $R$-modules over compact spaces, such that each
fiber is equipped with the compactly generated topology just recalled. We call such
such a gadget a \emph{quasi-bundle of finitely generated free $R$-modules}

\begin{theorem}\label{thm:loctriv}
Let $X$ be a compact space, $M$ a countable, torsion-free, seminormal,
cancellative monoid.  Let $E\to X$ be a quasi-bundle of finitely generated free $\C[M]$-modules. Assume that there exist an $n\ge 1$,
another quasi-bundle $E'$, and a quasi-bundle isomorphism $E\oplus E'\cong X\times \C[M]^n$. Then $E$ is locally trivial.
\end{theorem}
\begin{proof}
Put $R=\C[M]$. The isomorphism $E\oplus E'\cong X\times R^n$ gives a continuous function $e:X\to P_n(R)$; $E$ is locally
trivial if $e$ is locally conjugate to an idempotent of the form $1_r\oplus 0_\infty$, i.e. if it can be lifted
locally along the projection $\GL(R)\to P_n(R)$ to a continuous map $X\to \GL(R)$. Our hypothesis on
$M$ together with Theorems \ref{thm:pswan}, \ref{thm:algcomp}, \ref{thm:swan}, and \ref{thm:fib} imply that such local liftings exist.
\end{proof}

\section{Homotopy invariance.}\label{sec:htpy}

\subsection{From compact polyhedra to compact spaces: a result of Calder-Siegel.}\label{subsec:calder}

Consider the category $\comp$ of compact Hausdorff topological spaces with continuous maps and its full subcategory $\pol \subset \comp$ formed by those spaces which are compact polyhedra. In this subsection we show that for a functor which commutes with filtering colimits and is split exact on $C^*$-algebras,
homotopy invariance on $\pol$ implies homotopy invariance on $\comp$. For this we shall need a particular case of a result of Allan Calder and Jerrold Siegel \cite{calsie1, calsie2} that we shall presently
recall. We point out that the Calder-Siegel results have been further generalized by Armin Frei in \cite{frei}.
For each object
$X \in \comp$ we consider the comma category $(X \downarrow \pol)$ whose objects are morphisms
$f\colon X \to \cod(f)$ where the codomain $\cod(f)$ is a compact polyhedron. Morphisms are commutative diagrams as usual.
Let  $G\colon \pol \to \ab$ be a (contravariant) functor to the category of abelian groups. Its right Kan extension $G^\pol\colon \comp \to {\rm Ab}$ is defined by
$$G^\pol(X) = \colim_{f \in (X \downarrow \pol)} G(\cod(f)), \quad \forall X \in \comp.$$
The result of Calder-Siegel (see Corollary 2.7 and Theorem 2.8 in \cite{calsie2}) says that homotopy invariance properties of $G$ give rise to homotopy invariance properties of $G^{\pol}$. More precisely:

\begin{theorem}[Calder-Siegel] \label{thm:cs}
If $G\colon \pol \to \ab$ is a (contravariant) homotopy invariant functor, then the functor $G^{\pol}\colon \comp \to \ab$ is homotopy invariant.
\end{theorem}

We want to apply the theorem when $G$ is of the form $D\mapsto
E(C(D))$, the functor $E$ commutes with (algebraic) filtering
colimits, and is split exact on $C^*$-algebras. For this we have to
compare $E$ with the right Kan extension of $G$; we need some preliminaries.
Let $X\in\comp$, $D\subset\C$ the unit disk, and $\cF$ the set of all finite subsets
of $C(X,D)$. Since $X$ is compact, for each $f\in C(X)$ there is an
$n\in\N$ such that $(1/n)f\in C(X,D)$. Thus any finitely generated
subalgebra $A\subset C(X)$ is generated by a finite subset $F\subset
C(X,D)$. Let $\cF$ be the set of all finite subsets of $C(X,D)$;
since $C(X)$ is the colimit of its finitely generated subalgebras
$\C\langle F\rangle$, we have
\[
\colim_{F \in \cF} \C\langle F\rangle = C(X).
\]
For $F\in\cF$, write $Y_F\subset\C^F$ for the Zariski closure of the
image of the map $\alpha_F:X\to \C^F$, $x\mapsto (f(x))_{f\in F}$.
The image of $\alpha_F$ is contained in the compact semi-algebraic set
\[
P_F=D\cap Y_F
\]
In particular, $P_F$ is a compact polyhedron. Note that $\alpha_F$ induces an isomorphism between the ring
$\cO(Y_F)$ of regular polynomial functions and the subalgebra
$\C\langle F\rangle\subset C(X)$ generated by $F$. Hence the inclusion $P_F\subset Y_F$ induces a homomorphism
$\beta_F$ which makes the following diagram commute
\[
\xymatrix{\C\langle F\rangle\ar[dr]_{\beta_F}\ar[rr]&& C(X)\\
                &C(P_F)\ar[ur]&}
\]
Taking colimits we obtain

\begin{equation}\label{eq:fixem}
\xymatrix{C(X)\ar[dr]_{\beta}\ar@{=}[rr]&& C(X)\\
                &\colim_{F\in\cF}C(P_F)\ar[ur]^\pi&}
\end{equation}
Thus the map $\pi$ is a split surjection.

We can now draw the desired conclusion.

\begin{theorem}\label{thm:draw}
Let $E:\Co\to\ab$ be a functor. Assume that $F$ satisfies each of the following conditions.
\begin{enumerate}
\item $E$ commutes with filtered colimits.
\item $\pol\to\ab$, $D \mapsto E(C(D))$, is homotopy invariant.
\end{enumerate}
Then the functor
\[
\comp\to\ab,\quad X \mapsto E(C(X)),
\]
is homotopy invariant on the category of compact topological spaces.
\end{theorem}
\begin{proof}
By \eqref{eq:fixem} and the first hypothesis, the map $E(\beta)$ is
a right inverse of the map
$$E(\pi):E(\colim_{F\in\cF}C(P_F))=\colim_{F\in\cF}E(C(P_F))\to E(C(X)).$$ On the other
hand, we have a commutative diagram
\begin{equation}\label{diag:pi'}
\xymatrix{\colim_{F\in\cF}E(C(P_F))\ar[r]^(.45)\theta\ar[dr]_{E(\pi)}&\colim_{f
\in (X
\downarrow \pol)} E(C(\cod(f)))\ar[d]^{\pi'}\\
&E(C(X))}
\end{equation}
Hence the map $\pi'$ in the diagram above is split by the composite
$\theta E(\beta)$, and therefore $E(C(-))$ is naturally a direct
summand of the functor
\begin{equation}\label{lakan}
G^\pol:\comp\to\ab,\quad X\mapsto \colim_{f \in (X \downarrow \pol)}
E(C(\cod(f)))
\end{equation}
But \eqref{lakan} is the right Kan extension of the functor
$G:\pol\to \ab$, $G(K)=E(C(K))$, and thus it is homotopy invariant
by the second hypothesis and Calder-Siegel's theorem. It follows
that $E(C(-))$ is homotopy invariant, as we had to prove.
\end{proof}

\begin{remark} \label{rem:cstar}
If in Theorem \ref{thm:draw} the functor $E$ is split exact, then $A\mapsto
E(A)$ is homotopy invariant on the category of commutative
$C^*$-algebras. Indeed, since every commutative unital $C^*$-algebra
is of the form $C(X)$ for some compact space $X$, it follows that
$F$ is homotopy invariant on unital commutative $C^*$-algebras.
Using this and split exactness, we get that it is also homotopy
invariant on all commutative $C^*$-algebras.
\end{remark}

\begin{remark}
In general, one cannot expect that the homomorphism $\pi'$ of
\eqref{diag:pi'} be an isomorphism. For the injectivity one would
need the following implication: If $D$ is a compact polyhedron, and
$f \colon X \to D$ and $s\colon D \to \C$ are continuous maps such
that $0=s \circ f\colon X \to \C$, then there exist a compact
polyhedron $D'$ and continuous maps $g:X \to D'$, $h: D' \to D$ such
that $h \circ g = f \colon X \to D$ and $0 = h\circ s\colon D' \to
\C$. That is too strong if $X$ is a pathological space. To give a
concrete example: let $X$ be a Cantor set inside $[0,1]$, \ $f$ the
natural inclusion, and $s$ the distance function to the Cantor set,
and suppose $g$ and $h$ as above exist. If $0=h \circ s\colon D' \to
\C$, then the image of $D'$ in $[0,1]$ has to be contained in $X$.
But the image has only finitely many connected components, since
$D'$ has this property. Hence, since $X$ is totally disconnected,
the image of $D'$ in $[0,1]$ cannot be all of $X$. This is a
contradiction.
\end{remark}

\subsection{Second proof of Rosenberg's conjecture.}\label{subsec:second}

A second proof of Rosenberg's conjecture can be obtained by combining Theorem \ref{thm:draw} with
the following theorem, which is due to Eric Friedlander and Mark E. Walker.

\begin{theorem}[Theorem $5.1$ in \cite{fw}] \label{thm:fw}
If $n>0$ and $q\ge 0$, then
\[
K_{-n}(C(\Delta^q))=0
\]
\end{theorem}

\noindent\emph{Second proof of Rosenberg's conjecture.} By Proposition \ref{prop:homcor} and Theorem \ref{thm:draw}
it suffices to show that $K_n(C(D))=0$ for contractible $D\in\pol$. If $D$ is contractible, then the identity $1_D \colon D \to D$ factors over the cone $cD$. Hence, it is sufficient to show that $K_n(C(cD))=0$. The cone $cD$
is a star-like simplicial complex and for any subcomplexes $A,B \subset D$ with $A \cup B=D$, we get a Milnor square
$$
\xymatrix{ C(cD) \ar[r] \ar[d] & C(cA) \ar[d] \\ C(cB) \ar[r] & C(c(A \cap B)). }
$$
Since $cA$ is contractible, it retracts onto $c(A \cap B)$ and therefore, the square above is split. Using excision,
we obtain split-exact sequence of abelian groups as follows:
$$0 \to K_n(C(cD)) \to K_n(C(cA)) \oplus K_n(C(cB)) \to K_n(C(c(A\cap B))) \to 0.$$
Decomposing $cD$ like this, we see that the result of Theorem \ref{thm:fw} is sufficient for the vanishing of
$K_n(C(cD))$.
\qed

\subsection{The homotopy invariance theorem.}\label{subsec:htpy}

The aim of this subsection is to prove the following result.

\begin{theorem} \label{thm:htpy}
Let $F$ be a functor on the category of commutative $\C$-algebras with values in abelian groups. Assume that the following three conditions are satisfied.
\begin{itemize}
\item[(i)] $F$ is split-exact on $C^*$-algebras.
\item[(ii)] $F$ vanishes on coordinate rings of smooth affine varieties.
\item[(iii)] $F$ commutes with filtering colimits.
\end{itemize}
Then the functor
$$\comp\to\ab,\quad X \mapsto F(C(X))$$
is homotopy invariant on the category of compact Hausdorff topological spaces
and $F(C(X))=0$ for $X$ contractible.
\end{theorem}

\begin{proof}
Note that, since a point is a smooth algebraic variety, our hypothesis imply that $F(\C)=0$. Thus if $F$ is homotopy invariant and $X$ is contractible, then $F(X)=F(\C)=0$.
Let us prove then that $X \mapsto F(C(X))$ is homotopy invariant on the category of compact Hausdorff topological spaces.
Proceeding as in the proof of Theorem \ref{thm:fw}, we see that it is sufficient to show that $F(C(\Delta^n))=0$, for all $n \geq 0$.
Any finitely generated subalgebra of $C(\Delta^n)$ is reduced and hence corresponds to an algebraic variety over $\C$. Since $F$ commutes with filtered colimits, we obtain:
$$F(C(\Delta^n)) = \colim_{\Delta^n\to Y_{an}} \,F(\cO(Y)),$$
where the colimit runs over all continuous maps from $\Delta^n$ to the analytic variety $Y_{an}$ equipped
with the usual euclidean topology. For ease of notation, we will from now on just write $Y$ for both the
algebraic variety and the analytic variety associated to it. Let $\iota:\Delta^n\to Y$ be a continuous map.
As in the proof of the algebraic approximation theorem (\ref{thm:boundedapprox}), we consider Hironaka's desingularization $\pi:\tilde{Y}\to Y$ and Jouanoulou's affine bundle torsor $\sigma: Y'\to \tilde{Y}$.
Let $T \subset Y$ be a compact semi-algebraic subset such that $\iota(\Delta^n) \subset T$. Because $\pi$
is a proper morphism, $\tilde{T}=\pi^{-1}(T)$ is compact and semi-algebraic. By definition of vector bundle
 torsor (\cite{kh}), there is a Zariski cover of $\tilde{Y}$ such that the pull-back of $\sigma$ over each open subscheme $U\subset\tilde{Y}$ of the covering is isomorphic, as a scheme over $U$, to an algebraic trivial vector bundle. Thus $\sigma$ is locally trivial fibration for the euclidean topologies, and the trivialization maps are (semi)-algebraic. Hence because $\tilde{T}$, being compact, is locally compact, we may find a finite covering $\{\tilde{T}_i\}$ of $\tilde{T}$ by closed semi-algebraic subsets such that $\sigma$ is a trivial fibration over each $\tilde{T}_i$, and compact semi-algebraic subsets
$S_i \subset Y'$ such that $\sigma(S_i)=\tilde{T}_i$. Put $S=\bigcup_iS_i$; then $S$ is compact semi-algebraic,
and $f=(\pi\circ\sigma)|_S \colon S \to T$ is a continuous semi-algebraic surjection.
By Theorem \ref{thm:main}, there exists a semi-algebraic triangulation of $T$ such that $\ker(F(\Delta^m)\to F(f^{-1}(\Delta^m))=0$ for each simplex $\Delta^m$ in the triangulation. Consider the diagram
$$
\xymatrix{ F(C(f^{-1}(\Delta^m)) & F(C(S)) \ar[l] & F(\cO(Y')) \ar[l]\\
F(C(\Delta^m)) \ar[u] & F(C(T)) \ar[u] \ar[l] & F(\cO(Y)) \ar[l] \ar[u]}
$$
If $\alpha \in F(\cO(Y))$, then its image in $F(C(f^{-1}(\Delta^m)))$ vanishes since $f^{-1}(\Delta^m)\to \Delta^m$ factors through the smooth affine variety $Y'$, and $F(\cO(Y'))=0$. Hence, by Theorem \ref{thm:main}, $\alpha|_{\Delta^m}=0$, for each simplex in the triangulation.
Coming back to the map $\iota \colon \Delta^n \to Y$, we have a diagram
\[
\xymatrix{\Delta^n\ar[r]^\iota\ar[dr]^{\j}&Y\\
          \Delta^m\ar[r]_\theta & T\ar[u]}
\]
Here $\theta:\Delta^m\to T$ is the inclusion of a simplex in the triangulation, and $\j$ is the correstriction
of $\iota$. We need
to conclude that $\iota^*(\alpha)=0$, knowing only that $\theta^*(\alpha)=0$ for each simplex in a triangulation of $T$. This is done using split-exactness and barycentric subdivisions. Indeed, we perform the barycentric subdivision of $\Delta^n$ sufficiently many times so that each $n$-dimensional simplex is mapped to the closed star $\st(x)$ of some some vertex $x$ in the triangulation of $T$. Since $\Delta^n$ is star-like, the reduction argument of the proof of Theorem \ref{thm:boundedapprox} shows that it is enough to show the vanishing of $\iota^*(\alpha)$  for the (top-dimensional) simplices in this subdivision of $\Delta^n$. If ${\Delta'}^n$ is one of these top dimensional simplices, and $\Delta^m\subset\st(x)$, we can complete the diagram above to a diagram
\[
\xymatrix{{\Delta'}^n\ar[r]\ar[dr]&\Delta^n\ar[r]^\iota\ar[dr]^{\j}&Y\\
          \Delta^m \ar[r]& \st(x)\ar[r] & T\ar[u]}
\]
Hence it suffices to show that the pullback of $\alpha$ to $\st(x)$ vanishes. But since $\st(x)$ is
 star-like, then by the same reduction argument as before, the vanishing of the pullback of $\alpha$ to each of the top simplices $\Delta^m\subset \st(x)$ is sufficient to conclude that $\alpha_{|\st(x)}=0$. This finishes the proof.
\end{proof}

As an application, we obtain the following.

\goodbreak

\medskip

\paragraph{\emph{\bf Third proof of Rosenberg's conjecture.}}
\goodbreak
If $n<0$ then $K_n$ is split-exact and vanishes on coordinate rings of smooth affine algebraic varieties. By Theorem \ref{thm:htpy}, this implies that $X\mapsto K_n(C(X))$ is homotopy invariant.\qed

\subsection{A vanishing theorem for homology theories.}\label{subsec:vanish}

\begin{theorem}\label{thm:vanish}
Let $E:\Co/\C\to \Spt$ be a homology theory of commutative $\C$-algebras and $n_0\in\Z$. Assume
\begin{enumerate}
\item[(i)] $E$ is excisive on commutative $C^*$-algebras.
\item[(ii)] $E_n$ commutes with algebraic filtering colimits for $n\ge n_0$.
\item[(iii)] $E_n(\cO(V))=0$ for each smooth affine algebraic variety $V$ for $n\ge n_0$.
\end{enumerate}
Then $E_n(A)=0$ for every commutative $C^*$-algebra $A$ for $n\ge n_0$.
\end{theorem}

\begin{proof}
Let $n\ge n_0$. We have to show that $E_n(A)=0$ for every
commutative $C^*$-algebra $A$. Because by i), each $E_n$ is split
exact on commutative $C^*$-algebras, it suffices to show  that
$E_n(A)=0$ for unital $A$, i.e. for $A=C(X)$, $X\in\comp$. Since by
ii) $E_n$ preserves filtering colimits, the proof of Theorem \ref{thm:draw} shows that $E_n(C(X))$ is a direct summand
of
\[
\colim_{f
\in (X
\downarrow \pol)} E(C(\cod(f)))
\]
Hence it suffices to show that $E_n(C(D))=0$ for every
compact polyhedron $D$. By iii) and excision, this is true if  $\dim
D=0$. Let $m\ge 1$ and assume the assertion of the theorem holds for
compact polyhedra of dimension $<m$. By Theorem \ref{thm:htpy},
$D\mapsto E_n(C(D))$ is homotopy invariant; in particular
$E_n(C(\Delta^m))=0$. If $\dim D=m$ and $D$ is not a simplex, write
$D=\Delta^m\cup D'$, as the union of an $m$-simplex and a subcomplex
$D'$ which has fewer $m$-dimensional simplices. Put $L=\Delta^m\cap
D'$; then $\dim L<m$, and we have an exact sequence
\[
E_{n+1}(C(L))\to E_n(C(D))\to E_n(C(\Delta^m))\oplus E(C(D'))\to E_n(C(L))
\]
We have seen above that $E_n(C(\Delta^m))=0$; moreover $E_n(C(L))=E_{n+1}(C(L))=0$ because $\dim L<m$, and $E_n(C(D'))=0$ because $D'$ has fewer $m$-dimensional simplices than $D$. This concludes the proof.
\end{proof}

\section{Applications of the homotopy invariance and vanishing homology theorems.}\label{sec:appvanish}

\subsection{$K$-regularity for commutative $C^*$-algebras.}\label{subsec:kreg}

\begin{theorem}\label{thm:kreg}
Let $V$ be a smooth affine algebraic variety over $\C$, $R=\cO(V)$, and $A$ a commutative $C^*$-algebra. Then $A\otimes_\C R$ is $K$-regular.
\end{theorem}
\begin{proof}
For each fixed $p\ge 1$ and $i \in \Z$, write $F^p(A)=\cofi(K(A\otimes R)\to K(A[t_1,\dots,t_p]\otimes R))$ for the homotopy cofiber. It suffices
to prove that the homology theory $F^p:\Co/\C\to \Spt$ satisfies the hypothesis of Theorem \ref{thm:vanish}.
By \cite[Corollary 9.7]{MR997314}, $A[t_1,\dots,t_p]\otimes_\C R$ is $K$-excisive for every $C^*$-algebra
$A$ and every $p\ge 1$. It follows that the homology theory $F^p:\rA/\C\to \Spt$ is excisive on
$C^*$-algebras. In particular its restriction to $\Co/\C$ is excisive on commutative $C^*$-algebras.
Moreover, if $W$ is any smooth affine algebraic variety, then $R\otimes_\C\cO(W)=\cO(V\times W)$, is regular noetherian, and therefore $K$-regular. Finally, $F^p_*$ preserves filtering colimits because both $K_*$ and
$(-)\otimes\Z[t_1,\dots,t_p]\otimes R$ do.
\end{proof}

\begin{remark} \label{rem:creditreg}
The case $R=\C$ of the previous theorem was discovered by Jonathan Rosenberg. Unfortunately, the two proofs he has given, in \cite[Thm. 3.1]{rosop} and \cite[p.\ 866]{roshan} turned out to be problematic. A version
of Theorem \ref{thm:kreg} for $A=C(D)$, $D\in\pol$, was given by Friedlander and Walker in \cite[Thm.\ 5.3]{fw}.
Furthermore, Rosenberg acknowledges in \cite[p. 24]{roshan} that Walker also found a proof of this in the general case, but that he did not publish it. Anyhow, as Rosenberg observed in \cite[p. 91]{rosop},
in this situation, the polyhedral case implies the general case by a short reduction argument (this also follows from Theorem \ref{thm:draw} above).
Hence, an essentially complete argument for the proof of Theorem \ref{thm:kreg} existed already in the literature, although it was scattered in various
sources.
\end{remark}

The following corollary compares Quillen's algebraic $K$-theory with Charles A. Weibel's homotopy algebraic $K$-theory, $KH$, introduced in \cite{kh}.

\begin{corollary}\label{cor:correg}
If $A$ is a commutative $C^*$-algebra,
then the map  $K_*(A)\to KH_*(A)$ is an isomorphism.
\end{corollary}
\begin{proof} Weibel proved in \cite[Proposition 1.5]{kh} that if $A$ is a unital $K$-regular
 ring, then $K_*(A) \stackrel{\sim}{\to} KH_*(A)$. Using excision, it follows that this is true for all commutative $C^*$-algebras. Now apply Theorem \ref{thm:kreg}.
\end{proof}

\subsection{Hochschild and cyclic homology of commutative $C^*$-algebras.}\label{subsec:lq}

In the following paragraph we recall some basic facts about Hochschild and cyclic homology which we shall need; the standard reference for these topics is Jean-Louis Loday's book \cite{lod}.\\

Let $k$ be a field of characteristic zero. Recall a {\it mixed complex} of $k$-vectorspaces is a graded vectorspace $\{M_n\}_{n\ge 0}$
together with maps $b:M_*\to M_{*-1}$ and $B:M_*\to M_{*+1}$ satisfying $b^2=B^2=bB+Bb=0$. One can associate
various chain complexes to a mixed complex $M$, giving rise to the Hochschild, cyclic, negative cyclic and
periodic cyclic homologies of $M$, denoted respectively $HH_*$, $HC_*$, $HN_*$ and $HP_*$. For example
$HH_*(M)=H_*(M,b)$. A map of mixed complexes is a homogeneous map which commutes with both $b$ and $B$. It is
called a {\it quasi-isomorphism} if it induces an isomorphism at the level of Hochschild homology; this automatically implies that it also induces an isomorphism for $HC$ and all the other homologies mentioned above. For a $k$-algebra $A$ there is defined a mixed complex $(C(A/k),b,B)$ with
\[
C_n(A/k)=\left\{\begin{matrix}\tilde{A}_k\otimes_kA^{\otimes_kn}& n>0\\ A & n=0\end{matrix}\right.
\]
We write $HH_*(A/k)$, $HC_*(A/k)$, etc. for $HH_*(C(A/k))$, $HC_*(C(A/k))$, etc. If furthermore $A$ is unital and
$\bar{A}=A/k$, then there is also a mixed complex $\bar{C}(A/k)$ with $\bar{C}_n(A/k)=A\otimes_k\bar{A}^{\otimes_kn}$
and the natural surjection $C(A/k)\to \bar{C}(A/k)$ is a quasi-isomorphism. Note also that
\begin{equation}\label{cnuni}
\ker(\bar{C}(\tilde{A}_k/k)\to \bar{C}(k/k))=C(A/k)
\end{equation}
If $A$ commutative and unital, we have
a third mixed complex $(\Omega_{A/k},0,d)$
given in degree $n$ by $\Omega^n_{A/k}$, the module of $n$-K\"ahler differential forms, and where $d$ is the exterior derivation
of forms. A natural map of mixed complexes $\mu:\bar{C}(A/k)\to \Omega_{A/k}$ is defined
by
\begin{equation}\label{mumap}
\mu(a_0\otimes_k\bar{a}_1\otimes_k\dots\otimes_k\bar{a}_n)=\frac{1}{n!}a_0da_1\land\dots\land da_n
\end{equation}
It was shown by Loday and Quillen in \cite{lq} (using a classical result of Hochschild-Kostant-Rosenberg \cite{hkr}) that $\mu$ is a quasi-isomorphism if $A$ is
a \emph{smooth} $k$ algebra, i.e. $A=\cO(V)$ for some smooth affine algebraic variety over $k$. It follows from this (see \cite{lod}) that for
$Z\Omega^n_{A/k}=\ker(d:\Omega^n_{A/k}\to \Omega^{n+1}_{A/k})$ and $H_{dR}^*(A/k)=H^*(\Omega_{A/k},d)$,
we have $(n\in\Z)$
\begin{gather}\label{formu:lq}
HH_n(A/k)=\Omega^n_{A/k}\qquad  HC_n(A/k)=\Omega^n_{A/k}/d\Omega^{n-1}_{A/k}\oplus\bigoplus_{0\le 2i<n}H_{dR}^{n-2i}(A/k)\\
HN_n(A/k)=Z\Omega^n_{A/k}\oplus \prod_{p>0} H_{dR}^{n+2p}(A/k)\qquad HP_n(A/k)=\prod_{p\in\Z}H_{dR}^{2p-n}(A/k)\nonumber
\end{gather}

For arbitrary commutative unital $A$, there is
a decomposition \[C_n(A/k)=\oplus_{p=0}^nC^{(p)}(A/k)\] such that $b$ maps $C^{(p)}$ to itself, while
$B(C^{(p)})\subset C^{(p+1)}$ (see \cite{lod}). One defines
\[
HH^{(p)}_n(A/k)=H_nC^{(p)}(A/k)
\]
We have $HH_q^{(p)}(A/k)=0$ for $q<p$, $HH_p^{(p)}(A/k)=\Omega^p_{A/k}$, but in general for $q>0$,
$HH_{p+q}^{(p)}(A/k)\ne 0$. The map \eqref{mumap} is still a quasi-isomorphism
if $A$ is smooth over a field $F\supset k$; this follows from the Loday-Quillen result using the base change
spectral sequence of Kassel-Sletsj{\o}e, which we recall below.

\begin{lemma}\label{lem:main11} (Kassel-Sletsj{\o}e, \cite[4.3a]{ks})
Let $k\subseteq F$ be fields of characteristic zero.
For each $p\ge 1$ there is a bounded second quadrant homological
spectral sequence ($0\le i<p$, $j\ge0$):
\[
{}_p E^1_{-i,i+j}=\Omega^{i}_{F/k}\otimes_FHH^{(p-i)}_{p-i+j}(R/F)
\Rightarrow HH_{p+j}^{(p)}(R/k)
\]
\end{lemma}
\begin{corollary}\label{coro:ks}
If $A$ is a smooth $F$-algebra, then \eqref{mumap} is a quasi-isomorphism.
\end{corollary}

\begin{theorem}\label{thm:hh}
Let $X$ be a compact topological space, $A=C(X)$, and $k\subset\C$ a subfield.
Then the map \eqref{mumap} is a quasi-isomorphism, and we have the identities \eqref{formu:lq}.
\end{theorem}
\begin{proof}
Extend $C^{(p)}(/k)$ (and $HH_n^{(p)}(/k)$) to non-unital algebras by $$C^{(p)}(A/k)=\ker(C^{(p)}(\tilde{A}_k/k)\to C^{(p)}(k/k)).$$
Let $E^{(p)}(A/k)$ be the spectrum the Dold-Kan correspondence associates to $C^{(p)}(A/k)$. Regard $E^{(p)}$
as a homology theory of $\C$-algebras. Then $E^{(p)}$ is excisive on $C^*$-algebras, by Remark \ref{rem:wodpnas}
and naturality. Further $E^{(p)}_n(A/k)=HH_n^{(p)}(A/k)=0$ whenever $n>p$ and $A$ is smooth over $\C$, by Corollary \ref{coro:ks}.
It is also clear that $HH^{(p)}_*(/k)$ preserves filtering colimits, since $HH_*(/k)$ does. Thus we may apply
Theorem \ref{thm:vanish} to conclude the proof.
\end{proof}

\subsection{The Farrell-Jones Isomorphism Conjecture.}\label{subsec:fj}
Let $A$ be a ring, $\Gamma$ a group, and $q\le 0$. Put
\[
Wh^A_q(\Gamma)=\coker(K_q(A)\to K_q(A[\Gamma]))
\]
for the cokernel of the map induced by the natural inclusion $A\subset A[\Gamma]$.
Recall \cite[Conjecture ~1, pp. $708$]{MR2181833} that the Farrell-Jones conjecture with coefficients for a torsion free group implies that if $\Gamma$
torsion free and $A$ a noetherian regular unital ring, then
\begin{equation}\label{fj}
Wh_q^A(\Gamma)=0\qquad (q\le 0)
\end{equation}
Note that the conjecture in particular implies that $K_q(A[\Gamma])=0$ for $q<0$ if $A$ is noetherian regular,
for in this case we have $K_q(A)=0$ for $q<0$.

\begin{theorem}\label{thm:fj}
Let  $\Gamma$ be torsion free group which satisfies \eqref{fj} for every commutative smooth
$\C$-algebra $A$. Also let $q\le 0$. Then the functor $Wh^{?}_q(\Gamma)$ is homotopy invariant on commutative $C^*$-algebras.
\end{theorem}
\begin{proof} It follows from Theorem \ref{thm:htpy} applied to $Wh^{?}_q(\Gamma)$.
\end{proof}
\begin{corollary}\label{cor:fj}
Let $\Gamma$ be as above and $X$ a contractible compact space. Then
$K_0(C(X)[\Gamma])=\Z$ and $K_q(C(X)[\Gamma])=0$ for $q<0$.
\end{corollary}

\begin{corollary}
Let $\Gamma$ be as above. Then, the functor
$$X \mapsto K_q(C(X)[\Gamma])$$
is homotopy invariant on the category of compact topological spaces for $q \leq 0$.
\end{corollary}
\begin{proof}
This follows directly from Proposition \ref{prop:homcor} and the preceding corollary.
\end{proof}

The general case of the Farrell-Jones conjecture predicts that for any group $\Gamma$ and any unital ring $R$,
the assembly map \cite{MR2181833}
\begin{equation}\label{fjg}
\mathcal{A}^\Gamma (R):\mathbb{H}^{\Gamma}(E_{\mathcal{VC}}(\Gamma),K(R))\to K(R[\Gamma])
\end{equation}
is an equivalence. Here $\mathbb{H}^{\Gamma}(-,K(R))$ is the equivariant homology theory associated to the spectrum $K(R)$ and
$E_{\mathcal{VC}}((\Gamma))$ is the classifying space with respect to the class of virtually cyclic subgroups (see \cite{MR2181833} for definitions of these objects).

We also get:

\begin{theorem}\label{thm:fjg} Let $\Gamma$ be a group such that the map \eqref{fjg} is an equivalence for every smooth commutative $\C$-algebra
$A$. Then \eqref{fjg} is an equivalence for every $C^*$-algebra $A$.
\end{theorem}
\begin{proof}
It follows from Theorem \ref{thm:vanish} applied to $E(R)=\cofi(\mathcal{A}^\Gamma(R))$.
\end{proof}

\begin{remark}
We have $$K_q(C(X)[\Gamma]) = \pi^S_q\left(\map_+(X_+, KD(\C[\Gamma]))\right), \quad (q \leq 0)$$
where $KD(\C[\Gamma])$ denotes the diffeotopy $K$-theory spectrum, see Definition $4.1.3$ in \cite{MR2409415}. This follows from the study of a suitable co-assembly map and is not carried out in detail here. The homotopy groups of $KD(\C[\Gamma])$ can be computed from the equivariant connective $K$-homology of $\underbar{E} \Gamma$ using the Farrell-Jones assembly map.
\end{remark}

\subsection{Adams operations and the decomposition of rational $K$-theory.}\label{subsec:bs}

The rational $K$-theory of a unital commutative ring $A$ carries a natural decomposition
\[
K_n(A)\otimes\Q=\oplus_{i\ge 0}K_n(A)^{(i)}
\]
Here $K_n(A)^{(i)}=\bigcap_{k\ne 0}\{x \in K_n(A) \mid \psi^k(x)=k^i x\}$, where $\psi^k$ is the Adams operation.
For example, $K_0^{(0)}(A)=H^0(\spec A,\Q)$ is the rank component, and $K_n^{(0)}(A)=0$ for
$n>0$ (\cite[6.8]{kra}).
A conjecture of Alexander Be\u\i linson and Christophe Soul\'e (see \cite{bei}, \cite{sou}) asserts
that
\begin{equation}\label{bs}
K_n^{(i)}(A)=0 \text{ for } n\ge\max\{1,2i\}.
\end{equation}
The conjecture as stated was proved wrong for non-regular $A$ (see \cite{gw} and \cite[7.5.6]{ft}) but no regular counterexamples have been found. Moreover, the original statement has been formulated in terms of motivic cohomology (with rational,
torsion and integral coefficients) and generalized
to regular noetherian schemes \cite[4.3.4]{kahn}. For example if $X=\Spec R$ is smooth then
$K_n^{(i)}(R)=H^{2i-n}(X,\Q(i))$ is the motivic cohomology of $X$ with coefficients in the twisted sheaf $\Q(i)$.

We shall need the well-known fact that the validity of \eqref{bs} for $\C$ implies its validity for all smooth $\C$-algebras; this is Proposition \ref{prop:bs} below. In turn this uses the also well-known fact that
rational $K$-theory sends field inclusions to monomorphisms. We include proofs of both facts for completeness sake.

\begin{lemma}\label{lem:kinj}
Let $F\subset E$ be fields. Then $K_*(E)\otimes\Q\to K_*(F)\otimes\Q$ is injective.
\end{lemma}
\begin{proof} Since $K$-theory commutes with filtering colimits, we may assume that $E/F$ is a finitely generated field extension, which we may write as a finite extension of a finitely generated purely transcendental extension. If $E/F$ is purely transcendental, then by induction we are reduced to the case $E=F(t)$, which
follows from \cite[Thm. 1.3]{gers}. If $d=\dim_FE$ is finite, then the transfer map $K_*(E)\to K_*(F)$ \cite[pp. 111]{qui} splits $K_*(F)\to K_*(E)$ up to $d$-torsion.
\end{proof}

\begin{proposition}\label{prop:bs}
If \eqref{bs} holds for $\C$, then it holds for all smooth $\C$-algebras.
\end{proposition}
\begin{proof}
The Gysin sequence argument at the beginning of \cite[4.3.4]{kahn} shows that if \eqref{bs} is an isomorphism for all finitely generated field extensions of $\C$ then it is an isomorphism for all smooth $R$. If $E\supset \C$ is a finitely generated field extension, then we may write $E=F[\alpha]$ for some
purely transcendental field extension $F\cong \C(t_1,\dots,t_n)\supset \C$ and some algebraic element $\alpha$. From this and the fact that $\C$ is algebraically closed and of infinite transcendence degree over $\Q$, we see
that $E$ is isomorphic to a subfield of $\C$. Now apply Lemma \ref{lem:kinj}.
\end{proof}

\begin{theorem}\label{thm:bs}
Assume that \eqref{bs} holds for the field $\C$. Then it also holds for
all commutative $C^*$-algebras.
\end{theorem}
\begin{proof}
By Proposition \ref{prop:bs}, our current hypothesis imply that \eqref{bs} holds for smooth
$A$. In particular the homology theory $K^{(i)}$ vanishes on smooth $A$ for $n\ge n_0=\max\{2i,1\}$. Because
$K$-theory satisfies excision for $C^*$-algebras and commutes with algebraic filtering colimits,
the same is true of $K^{(i)}$. Hence we can apply Theorem \ref{thm:vanish}, concluding the proof.
\end{proof}

\begin{ack} Part of the research for this article was carried out during a visit of the first named author
to Universit\"at G\"ottingen. He is indebted to this institution for their hospitality. He also wishes to thank
Chuck Weibel for a useful e-mail discussion of Be\u \i linson-Soul\'e's conjecture. A previous version of this article contained a technical mistake in the proof of Theorem \ref{thm:draw}; we are thankful to Emanuel
Rodr\'\i guez Cirone for bringing this to our attention.
%, and Jimmy Petean for discussions which lead to Lemma \ref{lem:bound}.
\end{ack}

\bibliographystyle{plain}

\end{document}